\documentclass[a4paper, 11pt]{article}

\usepackage{amsmath,amssymb,amsthm}
\usepackage{graphicx}
\usepackage[british]{babel}
\usepackage[utf8]{inputenc}
\usepackage{xcolor}
\usepackage[colorlinks=true,linkcolor=blue,citecolor=red]{hyperref}
\usepackage{authblk}
\usepackage{epstopdf}
\usepackage{cite}
\usepackage{bm}
\usepackage{booktabs}
\usepackage[shortlabels]{enumitem}
\newtheorem{remark}{Remark}

\newtheorem{theorem}{Theorem}

\usepackage{mathtools}

\newcommand{\R}{\mathbb{R}}
\newcommand{\Rdz}{\mathbb{R}^{d_{\mathbf{z}}}}

\newcommand{\Ht}{\mathcal{H}}
\newcommand{\z}{\mathbf{z}}

\renewcommand{\epsilon}{\varepsilon}

\allowdisplaybreaks
\usepackage[nolist]{acronym}
\usepackage{autonum}
\usepackage{xcolor}
\usepackage[margin=1in]{geometry}

\begin{document}
\begin{acronym}
\acro{bkw}[BKW]{Bobylev-Krook-Wu}
\acro{gpc}[gPC]{generalised polynomial expansion}
\acro{sg}[sG]{stochastic Galerkin}
\end{acronym}

\title{\textbf{Uncertainty Quantification for the Homogeneous Landau-Fokker-Planck Equation via Deterministic Particle Galerkin methods}}

\author[1]{Rafael Bailo\thanks{\tt bailo@maths.ox.ac.uk}}
\author[1]{José A. Carrillo\thanks{\tt carrillo@maths.ox.ac.uk}}
\author[2]{Andrea Medaglia\thanks{\tt andrea.medaglia02@universitadipavia.it }}
\author[2]{Mattia Zanella\thanks{\tt mattia.zanella@unipv.it}}
\affil[1]{Mathematical Institute, University of Oxford, United Kingdom}
\affil[2]{Department of Mathematics ``F. Casorati'', University of Pavia, Italy}

\date{}

\maketitle

\abstract{
We design a deterministic particle method for the solution of the spatially homogeneous Landau equation with uncertainty. The deterministic particle approximation is based on the reformulation of the Landau equation as a formal gradient flow on the set of probability measures, whereas the propagation of uncertain quantities is computed by means of a \acf{sg} representation of each particle. This approach guarantees spectral accuracy in uncertainty space while preserving the fundamental structural properties of the model: the positivity of the solution, the conservation of invariant quantities, and the entropy production. We provide a regularity results for the particle method in the random space. We perform the numerical validation of the particle method in a wealth of test cases.
}
\\[+.2cm]

\section{Introduction}

This work introduces a deterministic particle scheme that is able to simultaneously preserve the physical properties of the spatially homogeneous Landau equation and quantify the evolution of uncertainty due to uncertain model parameters and initial conditions. The Landau equation is one of the most physically relevant models to describe plasma, which is itself a key ingredient in the development of nuclear fusion and the search for new energy sources. As such, the development of robust numerical tools in this area is of paramount importance.

In the absence of uncertainties, the evolution of the distribution function of charged particles in a plasma, $f$, in the phase space $(x,v)\in\R^d\times\R^d$, at time $t>0$, is governed by the Landau-Fokker-Planck equation:
\begin{equation}\label{eq:LFP}
\partial_t f(x,v,t) + v \cdot \nabla_x f(x,v,t) + F(x,v,t) \cdot \nabla_v f(x,v,t) = Q(f,f)(x,v,t)
\end{equation}
where $d\geq2$ is the dimension, and $F(\cdot)$ is a force acting on the particles, which may be either external or self-consistent. The right hand side $Q(f,f)$ is the Landau collision operator, a non-local term that describes the localised Coulomb interactions between charged particles. The operator is given by
\begin{equation}\label{eq:Coll_Op}
Q(f,f)(x,v,t) \coloneqq \nabla_v \cdot \int_{\mathbb R^d} A(v-v_*)\left[ f(v_*) \nabla_v f(v) - f(v) \nabla_{v_*}f(v_*) \right] \,\mathrm{d}v_*,
\end{equation}
where $A(\cdot)$ is a $d\times d$ symmetric and positive-semi-definite matrix that encodes the so-called \textit{collisional cross-section}:
\begin{equation}
A(q) = C |q|^{\gamma+2} S(q), \quad S(q) = I - \frac{q \otimes q}{|q|^2},
\end{equation}
where $C>0$ is the collision strength, $I$ is the $d\times d$ identity matrix, and $S(q)$ is the projection matrix onto the perpendicular of $q$.

The exponent $-d-1\leq\gamma\leq1$ governs the type of interaction: $\gamma>0$ corresponds to so-called \textit{hard potentials}, $\gamma<0$ to \textit{soft potentials}, and $\gamma=0$ is the Maxwell case. The Coulomb case, $\gamma=-d$, plays a crucial role in plasma physics, as it captures the Coulomb interactions between charged particles, and is deemed the most physically relevant choice.

This work focusses on the space-homogeneous Landau equation,
\begin{equation}\label{eq:L}
\partial_t f(v,t) = Q(f,f)(v,t)
\end{equation}
for $v\in\R^d$ and $t>0$; we briefly recall its properties here. The collision operator \eqref{eq:Coll_Op} can be rewritten (for sufficiently smooth solutions) as 
\begin{equation}
Q(f,f)(v,t) = \nabla_v \cdot \int_{\mathbb R^d} A(v-v_*)f(v)f(v_*)\left[ \nabla_v \log f(v) - \nabla_{v_*}\log f(v_*)\right]\,\mathrm{d}v_*.
\end{equation} 
The weak formulation of the equation can be expressed, after symmetrisation, as
\begin{equation}
\frac{\mathrm{d}}{\mathrm{d}t} \int_{\mathbb R^d} \phi f \,\mathrm{d}v
=
- \frac{1}{2} \iint_{\mathbb R^d \times \mathbb R^d}\!\!\!\!\!\!\!\!\!\!\! \left( \nabla_v\phi - \nabla_{v_*} \phi_* \right) \cdot A(v-v_*)f f_*\left[ \nabla_v \log f - \nabla_{v_*}\log f_*\right]\,\mathrm{d}v \,\mathrm{d}v_*,
\end{equation}
for any test function $\phi=\phi(v)$, with the notation $\phi_*=\phi(v_*)$, $f=f(v)$, and $f_*=f(v_*)$.

The conservation properties of the equation become apparent by choosing $\phi=1,v,|v|^2$. Conservation of mass ($\phi=1$) and momentum ($\phi=v$) are immediate, and conservation of energy ($\phi=|v|^2$) follows by observing the fact that, in that case, the matrix $A$ projects the vector $(\nabla_v \log f - \nabla_{v_*}\log f_*)$ onto the perpendicular of $(\nabla_v\phi - \nabla_{v_*} \phi_*)$. Therefore:
\begin{equation}
\frac{\mathrm{d}}{\mathrm{d}t} \int_{\mathbb R^d} \begin{pmatrix} 1 \\ v \\ |v|^2 \end{pmatrix} f \,\mathrm{d}v = 0.
\end{equation}

The dissipation properties of the equation are best understood by first defining the entropy functional
\begin{equation}
\Ht(f)(t) = \int_{\mathbb R^d} f(v,t) \log \left(f(v,t)\right) \,\mathrm{d}v.
\end{equation} 
A formal calculation shows the dissipation of the entropy,
\begin{equation}
\frac{\mathrm{d}}{\mathrm{d}t} \int_{\mathbb R^d} f \log f \,\mathrm{d}v \coloneqq - \mathcal{D}(f)(t) = - \frac{1}{2} \iint_{\mathbb R^d \times \mathbb R^d} b_{v,v_*} \cdot A(v-v_*) b_{v,v_*} f f_* \,\mathrm{d}v \,\mathrm{d}v_*\leq 0,
\end{equation}
where $b_{v,v_*}\coloneqq\nabla_v \log f - \nabla_{v_*} \log f_*$, and where the non-positivity follows from the positive-semi-definite property of the matrix $A$. This dissipation property can be used to characterise the kernel of $Q$, which is comprised exclusively of Maxwellians \cite{GZ2017}: distributions of the form
\begin{equation}
\mathcal{M}_{\rho,U,T} = \rho \left( \dfrac{1}{2\pi T } \right)^{\frac{d}{2}} \exp\left( - \dfrac{ (v - U)^2}{2 T}\right),
\end{equation}
for constants $\rho$ (the mass) and $T$ (the temperature), and constant vector $U$ (the mean velocity).

The mathematical study of both the homogeneous Landau equation and the Landau-Fokker-Planck is a challenging and active area of research. Many problems remain open, particularly in the case of soft potentials, which is the most physically relevant. Numerically, the challenges arise from the discretisation of the collision operator. Not only is the problem high-dimensional, leading to costly computations; it is also non-trivial to discretise $Q$ in a way that preserves the rich physical structure of the equations (the positivity of the solution, the conservation of invariant quantities, and the dissipation of the entropy functional).

Numerical methods for kinetic plasma equations can generally be classified into two categories. The first, methods based on a direct discretisation of the PDEs, like finite-differences and finite-volume methods \cite{Dimarco2015, Crouseilles2004, Filbet2001, xiao2021stochastic}, semi-Lagrangian schemes \cite{sonnendrucker1999semi, crouseilles2010}, and Fourier spectral methods \cite{Pareschi2000, gamba2017fast}. The second, methods based on a particle approximation of the distribution function, and consequently of the underlying dynamics, like particle-in-cell (PIC) methods \cite{birdsall85, pinto2014charge}, direct simulation Monte Carlo (DSMC) methods \cite{bobylev2000,RRCD,Dimarco2008,Dimarco2010}, or deterministic particle methods \cite{Jingwei2020, Carrillo2022}. 

The study, quantification, and control of the propagation of uncertainty in kinetic equations is a topic of undeniable relevance \cite{Pareschi2021,Jin2017,liu2018hypocoercivity}. This understanding is crucial for real-world applications, as knowledge about models is often incomplete. Uncertainty, whether it be in model parameters, initial conditions, or boundary conditions, can drastically affect model predictions, and their sensitivity must be quantified. 

Of course, the inclusion of random parameters poses, in turn, additional challenges in the theoretical and numerical analysis of our problem. From an analytical point of view, it is crucial to study the sensitivity of the solution with respect to the uncertainties in order to guarantee the regularity of the model, which will be of great importance in the design of numerical methods. As for the latter, they must also handle an increased dimensionality in the problem, which could prove computationally expensive.

In this work, we propose a hybrid \acf{sg} particle method for the numerical solution of the spatially homogeneous Landau equation with uncertainty. The scheme is based on a particle approximation of the equation in phase space, and a subsequent \ac{gpc} of the particles in the random parameter space. In this way, provided sufficient regularity, the method has spectral accuracy in uncertain space and preserves the structural properties of the model (the conservation of invariant quantities and the dissipation of the entropy), which would typically br lost with a standard \ac{sg} polynomial approximation scheme.

The particle approximation in phase space was first described in \cite{Jingwei2020}. The authors propose a new variational formulation that characterises the Landau equation as a gradient flow; the flow is formally a continuity equation with a velocity field defined by a regularisation of the entropy functional. In this way, the solution can be approximated by an empirical measure, giving rise to a coupled system of ordinary differential equations for the particles. It was shown that the particle method preserves the positivity, mass, momentum, and energy of the solution, as well as the dissipation of the entropy functional. For a formal discussion on the properties of the gradient flow reformulation, see also \cite{carrillo2022boltzmann, carrillo2023convergence}.

Once uncertainty is incorporated into the method, the particles are then approximated by their \ac{gpc} expansion in the space of the random parameters, thus obtaining a fully coupled system for the coefficients of this expansion. This approach was recently proposed for particle methods with applications to the homogeneous Boltzmann \cite{Pareschi2020} and Landau \cite{medaglia2023} equations, as well as to a Vlasov-Poisson-BGK system \cite{medaglia2022JCP}. Previously, the technique was used in multi-agent systems \cite{Carrillo19,CarrilloPZ19,Medaglia2022}.

The rest of the work is organised as follows. In Section \ref{sec:det} we introduce the deterministic particle method in the absence of uncertainties, recalling the relevant properties of the scheme. In Section \ref{sec:unc} we extend the method to the Landau equation with uncertainty, and we investigate the regularity of the scheme in the space of the random parameters. In Section \ref{sec:num} we validate the \ac{sg} particle method through several numerical tests; we verify the spectral accuracy of the scheme in the uncertain space, its agreement with the \ac{bkw} solution and Trubnikov's formula, and its trend to equilibrium.

\section{Deterministic particle method} \label{sec:det}

We begin by describing the deterministic particle method, which is based on the interpretation of the Landau equation as a gradient flow, and its regularisation.

\subsection{The Landau equation as a gradient flow}

The Landau equation can be interpreted as a (formal) gradient flow on the space of probability measures. This idea has been explored in recent works \cite{Jingwei2020, carrillo2022boltzmann}, which are in turned inspired by previous contributions on the non-linear Fokker-Planck \cite{ambrosio2005gradient,carrillo2003kinetic,carrillo2010nonlinear} and Boltzmann \cite{erbar2023gradient} equations.

This interpretation relies on the variation of the entropy functional (the $L^2$ representative of the Gateaux derivative of the functional) with respect to zero-mass perturbations, which is simply
\begin{equation}
\frac{\delta \Ht}{\delta f} = \log f.
\end{equation}
We may formally rewrite the collision operator \eqref{eq:Coll_Op} as
\begin{equation}\label{eq:Q_gr_fl}
Q(f,f) = \nabla_v \cdot \left( f \int_{\mathbb R^d} A(v-v_*) \left(\nabla_v\frac{\delta \Ht}{\delta f} - \nabla_{v_*}\frac{\delta \Ht_{*}}{\delta f_*}\right) f_* \,\mathrm{d}v_* \right),
\end{equation} 
leading to the Landau equation in \textit{continuity equation form},
\begin{equation}
\partial_t f(v,t) = - \nabla_v \cdot (f(v,t) U(f)(v,t)),
\end{equation}
where
\begin{align}
U(f)(v,t) &= - \int_{\mathbb R^d} A(v-v_*) b(v,v_*) f_* \,\mathrm{d}v_*,\\
b(v,v_*) &= \nabla_v\frac{\delta \Ht}{\delta f} - \nabla_{v_*}\frac{\delta \Ht_*}{\delta f_*}.
\end{align}
The corresponding weak formulation is
\begin{equation}
\int_{\mathbb R^d} \phi\, Q(f,f) \,\mathrm{d}v = 
- \frac{1}{2} \iint_{\mathbb R^{2d}} (\nabla_v \phi - \nabla_{v_*} \phi_*)
\cdot A(v-v_*)\left(\nabla_v\frac{\delta \Ht}{\delta f} - \nabla_{v_*}\frac{\delta \Ht_*}{\delta f_*}\right) f f_* \,\mathrm{d}v \,\mathrm{d}v_*.
\end{equation}

The advantage of this reformulation is that the entropy functional may now be regularised without altering the structural properties of the operator (the conservation of invariant quantities and the entropy dissipation). This feature is what permits the development of structure-preserving deterministic particle methods. 

\subsection{Two regularised gradient flows}

In this section, we study two different regularisations of the homogeneous Landau equation, arising from a symmetric and an anti-symmetric regularisation of the entropy functional. For a formal discussion on the properties of the regularised Landau equation, we refer to \cite{Jingwei2020, carrillo2022boltzmann}. We shall use the regularisations to motivate the deterministic particle methods, as well as their extension to the case with uncertainty in the sequel.

To construct the regularisations, we follow \cite{carrillo2019blob} in considering a Gaussian mollifier of variance $\epsilon>0$, \begin{equation}
\psi_\epsilon (v) = \frac{1}{(2\pi\epsilon)^{\frac{d}{2}}}\exp\left(-\frac{|v|^2}{2\epsilon}\right)
\end{equation}
and then we propose:
\begin{enumerate}[(a)]
\item a \emph{symmetric} regularisation of the entropy functional,
\begin{equation}\label{eq:simm}
\Ht_{\epsilon}(f)(t) = \int_{\mathbb R^d} (f\ast \psi_\epsilon) \log(f\ast \psi_\epsilon) \,\mathrm{d}v;
\end{equation}
\item an \emph{anti-symmetric} regularisation,
\begin{equation}\label{eq:antisimm}
\Ht_{\epsilon}(f)(t) = \int_{\mathbb R^d} f \log(f\ast \psi_\epsilon) \,\mathrm{d}v.
\end{equation}
\end{enumerate}
Other mollifiers are investigated in \cite{carrillo2023convergence}, but we retain the Gaussian choice in this chapter for simplicity.

Employing one of the regularised entropies, we arrive at the regularised gradient flow:
\begin{equation}\label{eq:Landau_reg}
\partial_t f(t,v) = Q_\epsilon(f,f) \coloneqq - \nabla_v\cdot(fU_\epsilon(f)),
\end{equation}
where 
\begin{align}
U_\epsilon(f)(t,v) &= - \int_{\mathbb R^d} A(v-v_*) b_\epsilon(v,v_*) f_* \,\mathrm{d}v_*,\\
b_\epsilon(v,v_*) &= \nabla_v \frac{\delta \Ht_\epsilon}{\delta f} - \nabla_{v_*} \frac{\delta \Ht_{\epsilon,*}}{\delta f_*}.
\end{align}
The variation and its gradient are as follows:
\begin{enumerate}[(a)]
\item in the symmetric case \eqref{eq:simm}, we obtain
\begin{equation}
\frac{\delta\Ht_\epsilon}{\delta f} = \log(f\ast\psi_\epsilon)\ast\psi_\epsilon,
\quad \nabla_v\frac{\delta\Ht_\epsilon}{\delta f} = \log(f\ast\psi_\epsilon)\ast\nabla_v\psi_\epsilon;
\end{equation}
\item in the anti-symmetric case \eqref{eq:antisimm}, we obtain
\begin{equation}
\frac{\delta\Ht_\epsilon}{\delta f} = \log(f\ast\psi_\epsilon) + \frac{f}{f\ast\psi_\epsilon}\ast\psi_\epsilon, \quad \nabla_v\frac{\delta\Ht_\epsilon}{\delta f} = \frac{f\ast\nabla_v\psi_\epsilon}{f\ast\psi_\epsilon} + \frac{f}{f\ast\psi_\epsilon}\ast\nabla_v\psi_\epsilon.
\end{equation}
\end{enumerate}

\subsection{Deterministic particle method}
Both regularisations \eqref{eq:simm}-\eqref{eq:antisimm} of the homogeneous Landau equation will lead to particle methods with suitable properties. The methods arise by considering the regularised flow \eqref{eq:Landau_reg} as a continuity equation for an empirical solution
\begin{equation} \label{eq:emp}
f^N(v,t) = \sum_{i=1}^{N} w_i \delta(v-v_i(t)),
\end{equation}
where $w_i>0$ is the weight of the particle $i$ for every $i=1,\dots,N$. A smoothed solution (\textit{blob solution}) can be then defined as
\begin{equation}\label{eq:reg_discr}
\tilde{f}^N(v,t) \coloneqq (f^N\ast\psi_\epsilon)(v,t) = \sum_{i=1}^{N} w_i \psi_\epsilon(v-v_i(t)).
\end{equation}
Within the particle interpretation, $\psi_\epsilon$ is akin to the ``shape'' of the particles.

Substituting \eqref{eq:emp} into \eqref{eq:Landau_reg} as a distributional solution to the regularised Landau equation, we obtain the time evolution rule of the system of $N$ particles
\begin{align}
\label{eq:det_scheme}
\frac{\mathrm{d} v_i(t)}{\mathrm{d}t} = U_\epsilon(f^N)(v_i,t) &= - \sum_{j=1}^N w_j A(v_i-v_j) b^N_\epsilon(v_i,v_j) \\
& = - \sum_{j=1}^N w_j A(v_i-v_j)\left(\nabla\frac{\delta \Ht^{N}_{\epsilon}}{\delta f}(v_i) - \nabla\frac{\delta \Ht^{N}_{\epsilon}}{\delta f}(v_j) \right).
\end{align}
The variation of the gradient becomes:
\begin{enumerate}[(a)]
\item for the symmetric regularisation \eqref{eq:simm},
\begin{equation}
\nabla\frac{\delta \Ht^{N}_{\epsilon}}{\delta f}(v_i) = \int_{\mathbb R^d} \nabla\psi_\epsilon(v-v_i) \log\left( \tilde{f}^N (v,t) \right) \,\mathrm{d}v;
\end{equation}
\item for the anti-symmetric regularisation \eqref{eq:antisimm},
\begin{equation}
\nabla\frac{\delta \Ht^{N}_{\epsilon}}{\delta f}(v_i) = \frac{\nabla\tilde{f}^N(v_i)}{\tilde{f}^N(v_i)} + \sum_{k=1}^N w_k \frac{\nabla\psi_\epsilon(v_i-v_k)}{\tilde{f}^N(v_k)}.
\end{equation}
\end{enumerate}
The time dependency of $\tilde{f}^N$, $\nabla\frac{\delta \Ht^{N}_{\epsilon}}{\delta f}$, and $ b^N_\epsilon$ is omitted for simplicity.

\begin{remark}[Semi-discrete vs discrete schemes in velocity]
The symmetric regularisation leads only to a semi-discrete--in--velocity scheme, since the evaluation of the term $\nabla_v\frac{\delta\Ht_\epsilon}{\delta f}$ requires quadrature. Meanwhile, the anti-symmetric discretisation is more convenient in practice, since it leads to a fully discrete-in-velocity scheme; every convolution that appears in the method involves an empirical distribution, and thus become a discrete sum.
\end{remark}

The particle method is structure-preserving with either discretisation:

\begin{theorem}[Properties of the semi-discrete--in--velocity scheme]\label{th:prop_det_simm}
The deterministic particle method~\eqref{eq:det_scheme} with the symmetric regularisation \eqref{eq:simm} preserves the structural properties of the Landau equation: 
\begin{enumerate}
\item Conservation of mass, momentum, and energy 
\begin{equation}
\frac{\mathrm{d}}{\mathrm{d}t} \sum_{i=1}^{N} w_i = 0,
\quad
\frac{\mathrm{d}}{\mathrm{d}t} \sum_{i=1}^{N} w_i v_i(t) = 0, 
\quad
\frac{\mathrm{d}}{\mathrm{d}t} \sum_{i=1}^{N} w_i |v_i(t)|^2= 0.
\end{equation}
\item Dissipation of the semidiscrete entropy of the solution
\begin{equation}
\frac{\mathrm{d}}{\mathrm{d}t} \Ht^N_\epsilon = - \mathcal{D}^N_\epsilon \leq 0
\end{equation}
where
\begin{align}
\Ht^N_\epsilon &= \int_{\mathbb R^d} (f^N \ast \psi_\epsilon) \log\left(f^N \ast \psi_\epsilon\right) \,\mathrm{d}v \\
\mathcal{D}^N_\epsilon &= \frac{1}{2} \sum_{i,j=1}^N w_i w_j b(v_i,v_j) \cdot A(v_i-v_j) b(v_i,v_j).
\end{align}
\end{enumerate} 
\end{theorem}

\begin{proof}
For the explicit computations, we refer to \cite{Jingwei2020}.
\end{proof}

\begin{theorem}[Properties of the discrete-in-velocity scheme]\label{th:prop_det}
The deterministic particle method~\eqref{eq:det_scheme} with the anti-symmetric regularisation \eqref{eq:antisimm} preserves the structural properties of the Landau equation: 
\begin{enumerate}
\item Conservation of mass, momentum, and energy 
\begin{equation}
\frac{\mathrm{d}}{\mathrm{d}t} \sum_{i=1}^{N} w_i = 0,
\quad
\frac{\mathrm{d}}{\mathrm{d}t} \sum_{i=1}^{N} w_i v_i(t) = 0, 
\quad
\frac{\mathrm{d}}{\mathrm{d}t} \sum_{i=1}^{N} w_i |v_i(t)|^2= 0.
\end{equation}
\item Dissipation of the discrete entropy of the solution
\begin{equation}
\frac{\mathrm{d}}{\mathrm{d}t} \Ht^N_\epsilon = - \mathcal{D}^N_\epsilon \leq 0
\end{equation}
where
\begin{align}
\Ht^N_\epsilon &= \sum_{i=1}^N w_i \log\left( \sum_{j=1}^N w_j \psi_\epsilon(v_i-v_j) \right) \\
\mathcal{D}^N_\epsilon &= \frac{1}{2} \sum_{i,j=1}^N w_i w_j b(v_i,v_j) \cdot A(v_i-v_j) b(v_i,v_j).
\end{align}
\end{enumerate} 
\end{theorem}

\begin{proof}
To prove the conservations, we write the time evolution of any test function $\phi=\phi(v)$ depending on the particles:
\begin{align}
\frac{\mathrm{d}}{\mathrm{d}t} \sum_{i=1} w_i \phi(v_i) & = \sum_{i=1}^N w_i \nabla_v \phi(v_i) \cdot \frac{\mathrm{d}}{\mathrm{d}t} v_i \\
& = - \sum_{i,j=1}^N w_i w_j \nabla_v \phi(v_i) \cdot A(v_i-v_j) b(v_i,v_j) \\
& = - \frac{1}{2} \sum_{i,j=1}^N w_i w_j \left( \nabla_v \phi(v_i) - \nabla_{v} \phi(v_j) \right) \cdot A(v_i-v_j) b(v_i,v_j),
\end{align}
invoking the symmetry of $A$ and the antisymmetry of $b$: $A(v) = A(-v)$ and $b(v,v_*) = -b(v_*,v)$. The last line of the previous computation vanishes identically for $\phi(v)=1,v,|v|^2$.

To prove the dissipation, we compute:
\begin{align}
\frac{\mathrm{d}}{\mathrm{d}t} \Ht^N_\epsilon & = \sum_{i=1}^N w_i \frac{\sum_{j=1}^N w_j \nabla_{v} \psi_\epsilon(v_i-v_j) \cdot \frac{\mathrm{d}}{\mathrm{d}t}(v_i-v_j)}{\sum_{k=1}^{N}w_k\psi_\epsilon(v_i-v_k)} \\
& = \sum_{i,j=1}^N \frac{w_i w_j \nabla_{v} \psi_\epsilon(v_i-v_j)}{\sum_{k=1}^{N}w_k\psi_\epsilon(v_i-v_k)} \cdot \frac{\mathrm{d}}{\mathrm{d}t}(v_i-v_j) \\
& = \sum_{i=1}^N w_i \left( \frac{\sum_{j=1}^N w_j \nabla_{v} \psi_\epsilon(v_i-v_j) }{\sum_{k=1}^{N}w_k\psi_\epsilon(v_i-v_k)} + \sum_{j=1}^N \frac{ w_j \nabla_{v} \psi_\epsilon(v_i-v_j) }{\sum_{k=1}^{N}w_k\psi_\epsilon(v_i-v_k)} \right) \cdot \frac{\mathrm{d}}{\mathrm{d}t} v_i \\
& = \sum_{i=1}^N w_i \nabla_{v} \frac{\delta \Ht^N_\epsilon}{\delta f} (v_i) \cdot \frac{\mathrm{d}}{\mathrm{d}t} v_i \\
& = \sum_{i,j=1}^N w_i w_j \nabla_{v} \frac{\delta \Ht^N_\epsilon}{\delta f} (v_i) \cdot A(v_i-v_j) b(v_i,v_j) \\
& = - \frac{1}{2} \sum_{i,j=1}^N w_i w_j \left( \nabla_{v} \frac{\delta \Ht^N_\epsilon}{\delta f} (v_i) - \nabla_{v} \frac{\delta \Ht^N_\epsilon}{\delta f} (v_j) \right) \cdot A(v_i-v_j) b(v_i,v_j) \\
& = - \frac{1}{2} \sum_{i,j=1}^N w_i w_j b(v_i,v_j) \cdot A(v_i-v_j) b(v_i,v_j) \leq 0,
\end{align}
where we have swapped the $i$ and $j$ labels on the right summand, and have exploited the anti-symmetry of $\nabla_{v} \psi_\epsilon(v)$. 
The last line is non-positive because the matrix $A$ is positive semi-definite.
\end{proof}

\section{Stochastic Galerkin particle method} \label{sec:unc}

This section extends the particle method to handle uncertainty in the initial condition and parameters of the Landau equation.

\subsection{The Landau equation with uncertainty}

We consider now the space homogeneous Landau equation with uncertainty
\begin{align}
\partial_t f(v,t,\z)
& = - \nabla_v\cdot(f(v,t,\z)U(f)(v,t,\z))\\
& = \nabla_v \cdot \int_{\mathbb R^d} A(v-v_*,\z)\left[ \nabla_v f(v,\z) f(v_*,\z) - \nabla_{v_*}f(v_*,\z)f(v,\z)\right]\,\mathrm{d}v_*,
\end{align}
where $\z = (z_1,\dots,z_{d_{\z}})\in\Rdz$ is a random vector containing all the unknowns of the system. We assume $\z$ has a given distribution $p(\z)$ such that
\begin{equation}
\textrm{Prob}(\z\in I_{\z}) = \int_{I_{\z}} p(\z) \,\mathrm{d} \z 
\end{equation}
for any $I_{\z}\subseteq\Rdz$. The vector $\z$ characterises the missing information on the system due, for instance, to uncertain initial conditions, or to uncertain model parameters (e.g. the exponent of the interaction).

Following the previous notation, the regularised counterpart of the equation is
\begin{equation}\label{eq:LFP_unc}
\partial_t f(v,t,\z) = - \nabla_v\cdot(f(v,t,\z)U_\epsilon(f)(v,t,\z)).
\end{equation}
As in \eqref{eq:emp}, we consider a set of $N$ particles, defining the empirical distribution
\begin{equation}
f^N(v,t,\z) = \sum_{i=1}^N w_i(\z) \delta(v-v_i(t,\z));
\end{equation}
analogously, we introduce the smoothed solution
\begin{equation}
\tilde{f}^N(v,t,\z) \coloneqq (f^N \ast \psi_\epsilon) (v,t,\z) = \sum_{i=1}^N w_i(\z) \psi_\epsilon(v-v_i(t,\z)).
\end{equation}
We will assume in the following that the weights $w_i>0$ are independent from the uncertain vector $\z$, considering only the velocities as functions of $\z$.

A particle method for \eqref{eq:LFP_unc} reads
\begin{align}\label{eq:unc_scheme}
\frac{\mathrm{d} v_i(t,\z)}{\mathrm{d}t} &= U_\epsilon(f^N)(v_i,t,\z) = - \sum_{j=1}^N w_j A_{ij}(t,\z) b^N_{\epsilon,ij}(t,\z) \\
&= -\sum_{j=1}^N w_j A_{ij}(t,\z)\left(\nabla\frac{\delta \Ht^{N}_{\epsilon}}{\delta f}(v_i(\z)) - \nabla\frac{\delta \Ht^{N}_{\epsilon}}{\delta f}(v_j(\z)) \right),
\end{align}
where we have denoted 
\begin{equation}
A_{ij}(t,\z)=A(v_i(t,\z)-v_j(t,\z))
\quad\text{and}\quad
b^N_{\epsilon,ij}(t,\z)=b^N_{\epsilon}(v_i(t,\z),v_j(t,\z)).
\end{equation}
Again, the term $b^N_{\epsilon,ij}(t,\z)$ can be given by the different regularisations:
\begin{enumerate}[(a)]
\item symmetric,
\begin{equation}
\nabla \frac{\delta \Ht^{N}_\epsilon}{\delta f}(v_i(\z)) = \int_{\mathbb R^d} \nabla\psi_\epsilon(v_i(\z)-v) \log\left( \tilde{f}^{N}(v,\z) \right) \,\mathrm{d}v; 
\end{equation}
\item anti-symmetric,
\begin{equation}
\nabla \frac{\delta \Ht^{N}_\epsilon}{\delta f}(v_i(\z)) = \frac{\nabla\tilde{f}^{N}(v_i(\z))}{\tilde{f}^{N}(v_i(\z))} + \sum_{j=1}^N w_j \frac{\nabla\psi_\epsilon(v_i(\z)-v_j(\z))}{\tilde{f}^{N}(v_j(\z))}.
\end{equation}
\end{enumerate}

\begin{remark} \label{th:remark1}
The \ac{sg} particle method \eqref{eq:unc_scheme} retains the properties of the deterministic method; namely, the conservation of the invariant quantities and the dissipation of the entropy. This follows immediately from the observation that Theorems \ref{th:prop_det_simm}-\ref{th:prop_det} hold pointwise in $\z$.
\end{remark}

\subsection{Stochastic Galerkin projection of the particle method}
In the following, we will construct a \acf{gpc} expansion of the particle method \eqref{eq:unc_scheme} in the space of the random parameters. A similar approach, based on a Monte Carlo approximation of the distribution function, has been proposed for the homogeneous Boltzmann \cite{Pareschi2020} and Landau \cite{medaglia2023} equations, and for a Vlasov-Poisson-BGK system \cite{medaglia2022JCP}. Other applications can be found in \cite{Carrillo19,CarrilloPZ19,Medaglia2022}.

In order to perform an \ac{sg} expansion of the scheme, we choose a set of $M+1$ polynomials of degree at most $M$, $\varPsi=\{\Psi_m(\z)\}_{m=0}^M$, orthonormal with respect to the distribution $p(\z)$:
\begin{equation}
\int_{I_{\z}} \Psi_m(\z) \Psi_n(\z) p(\z) \,\mathrm{d} \z = \mathbb{E}_{\z}[\Psi_m(\cdot)\Psi_n(\cdot)] = \delta_{mn}.	
\end{equation}
The velocity of each particle can then be approximated by its projection onto the span of $\varPsi$
\begin{equation} \label{eq:gPC}
v_i(t,\z) \approx v^M_i(t,\z) = \sum_{m=0}^{M} \hat{v}_{i,m}(t) \Psi_m(\z),
\end{equation}
where the coefficient of the expansion is
\begin{equation}
\hat{v}_{i,m}(t) \coloneqq \int_{I_{\z}} v_i(t,\z) \Psi_m(\z) p(\z) \,\mathrm{d} \z = \mathbb{E}_{\z}[v_i(t,\cdot)\Psi_m(\cdot)].
\end{equation}
The corresponding empirical and smoothed distributions are denoted respectively by
\begin{equation}
f^{N,M}(v,t,\z) = \sum_{i=1}^N w_i \delta(v-v^M_i(t,\z))
\quad\text{and}\quad
\tilde{f}^{N,M}(v,t,\z) = \sum_{i=1}^N w_i \psi_\epsilon(v-v^M_i(t,\z)).
\end{equation}
The \ac{sg} particle method is found by substituting the \ac{gpc} approximation $v^M_i(t,\z)$ into \eqref{eq:unc_scheme}, and then projecting the scheme onto the span of $\varPsi$:
\begin{align} 
\frac{\mathrm{d}}{\mathrm{d}t} \hat{v}_{i,m}(t) &= \int_{I_{\z}} U_\epsilon(f^{N,M})(v_i,t,\z) \Psi_m(\z) p(\z) \,\mathrm{d} \z, \label{eq:particleSG}\\
U_\epsilon(f^{N,M})(v_i,t,\z) &= - \sum_{j=1}^N w_j A^M_{ij}(t,\z) b^{N,M}_{\epsilon,ij}(t,\z), \\
b^{N,M}_{\epsilon,ij}(t,\z) &= \nabla\frac{\delta \Ht^{N}_{\epsilon}}{\delta f}(v^M_i(\z)) - \nabla\frac{\delta \Ht^{N}_{\epsilon}}{\delta f}(v^M_j(\z)),
\end{align}
where $A^M_{ij}(t,\z)=A(v^M_i(t,\z)-v^M_j(t,\z))$ and $b^{N,M}_{\epsilon,ij}(t,\z)$; once again there is a choice of regularisation:
\begin{enumerate}[(a)]
\item symmetric regularisation,
\begin{equation}
\nabla\frac{\delta \Ht^{N}_{\epsilon}}{\delta f}(v^M_i(\z)) = \int_{\mathbb R^d} \nabla\psi_\epsilon(v-v^M_i(\z)) \log\left( \tilde{f}^N (v,t) \right) \,\mathrm{d}v;
\end{equation}
\item anti-symmetric regularisation,
\begin{equation}
\nabla\frac{\delta \Ht^{N}_{\epsilon}}{\delta f}(v^M_i(\z)) = \frac{\nabla\tilde{f}^N(v^M_i(\z))}{\tilde{f}^N(v^M_i(\z))} + \sum_{k=1}^N w_k \frac{\nabla\psi_\epsilon(v^M_i(\z)-v^M_k(\z))}{\tilde{f}^N(v^M_k(\z))}.
\end{equation}
\end{enumerate}
The method \eqref{eq:particleSG} is a fully coupled system for the projections $\{\hat{v}_{i,m},\,i=1,\dots,N,\,m=0,\dots,M\}$, with computational complexity $O(N^2M^2)$. 

In practice, the integral in \eqref{eq:particleSG} will be computed using Gaussian quadrature. Once the $L$ nodes $\{z_l\}_{l=1}^{L}$ and $L$ weights $\{\omega_l\}_{l=1}^{L}$ have been chosen, the scheme reduces to
\begin{equation}\label{eq:particlesG_quad}
\frac{\mathrm{d}}{\mathrm{d}t} \hat{v}_{i,m}(t) \approx \sum_{l=1}^L U_\epsilon(f^{N,M})(v_i,t,z_l) \Psi_m(z_l) \omega_l.
\end{equation}

\subsection{Regularity in the space of the random parameters}
In Remark \ref{th:remark1}, we have observed the properties of the \ac{sg} method for a fixed parameter $\z\in\Rdz$. Here, we investigate the regularity of the discrete particle method \eqref{eq:unc_scheme} in the space of the random parameters. 

\begin{theorem}\label{th:vz}
Consider $\{v_i(t,\z)\}_{i=1}^N$, a particle solution to the regularised Lan\-dau-Fokker-Planck equation \eqref{eq:LFP_unc} at time $t\geq0$, following scheme \eqref{eq:unc_scheme}. We have
\begin{equation}
\sum_{i=1}^N w_i\| v_i(t,\z) \|^2_{L^2(I_{\z})} = 	\sum_{i=1}^N w_i \| v_i(0,\z) \|^2_{L^2(I_{\z})}.
\end{equation}
\end{theorem}
\begin{proof}
To make explicitly the computation, we consider the time evolution of any test function $\phi$:
\begin{equation} \label{eq:variation}
\frac{\mathrm{d}}{\mathrm{d}t} \sum_{i=1}^{N} w_i \phi(v_i(\z)) = - \frac{1}{2} \sum_{i,j=1}^N w_i w_j \left( \nabla\phi(v_i(\z)) - \nabla\phi(v_j(\z))\right) 
\cdot \left( A_{ij}(t,\z) b^N_{\epsilon,ij}(t,\z) \right),
\end{equation}
where we have omitted the time dependency of $v_i$ for the sake of brevity.
If we choose $\phi(v_i(t,\z)) = |v_i(t,\z)|^2$, and then we integrate in $I_{\z}$ against $p(\z) \,\mathrm{d} \z$, we obtain
\begin{equation}
\sum_{i=1}^N w_i \frac{\mathrm{d}}{\mathrm{d}t} \| v_i(t,\z) \|^2_{L^2(\Omega)} = 0,
\end{equation}
exploiting the conservation of energy result proven in the previous section.
\end{proof}

\begin{theorem} \label{th:dzvz}
Consider $\{v_i(t,\z)\}_{i=1}^N$, a particle solution to the regularised Lan\-dau-Fokker-Planck equation \eqref{eq:LFP_unc} at time $t\geq0$, following the scheme \eqref{eq:unc_scheme}. Consider also the constant $0\leq C_A < + \infty$ such that
\begin{equation}
\|\partial^k_{\z} A_{ij}(t,\z) \|_{L^\infty(I_\z)} \leq C_A,
\quad\text{for }
k=0,1,
\end{equation}
and the constants 
$0\leq C_B < + \infty $ and $0\leq C_{\partial_{\z}D} < + \infty$ such that
\begin{equation}
C_B = 4\frac{\log^2 N}{\epsilon^2} \left(\frac{2\epsilon}{\pi}\right)^{d}, \qquad C_{\partial_{\z}D} = \frac{\log(N)}{\epsilon}\left(1+\epsilon^{d-1}\right) .
\end{equation}
Defining $K_1 \coloneqq 2 C_A (1+ 3C_{\partial_{\z}D})$ and $K_2 \coloneqq 2 N C_A C_B$,
we find
\begin{equation}
\sum_{i=1}^{N} \| \partial_{\z} v_i(t,\z) \|^2_{L^2(I_\z)} \leq e^{ K_1 t} \sum_{i=1}^{N} \| \partial_{\z} v_i(0,\z) \|^2_{L^2(I_\z)} + K_2 t.
\end{equation}
\end{theorem}
\begin{proof}
We first take the $\z$ derivative of the particle method \eqref{eq:unc_scheme}; then we multiply by $2\partial_{\z}v_i(t,\z) p(\z) \,\mathrm{d} \z$ and integrate over $I_\z$ to obtain
\begin{align}
\frac{\mathrm{d}}{\mathrm{d}t} \| \partial_{\z} v_i(t,\z) \|^2_{L^2(I_{\z})} & = - 2 \int_{I_{\z}} \sum_{j=1}^N w_j \partial_{\z} v_i(t,\z) \cdot \partial_{\z}\left( A_{ij}(t,\z) b^N_{\epsilon,ij}(t,\z) \right) p(\z) \,\mathrm{d} \z \\
& \leq 2 \sum_{j=1}^N w_j \int_{I_{\z}} |\partial_{\z} v_i(t,\z) \cdot \partial_{\z}\left( A_{ij}(t,\z) b^N_{\epsilon,ij}(t,\z) \right) | p(\z) \,\mathrm{d} \z \\
& \leq 2 C_A \sum_{j=1}^N w_j \int_{I_{\z}} \left(|\partial_{\z} v_i(t,\z)| |b^N_{\epsilon,ij}(t,\z)| \right.\\
&\qquad\qquad \qquad\qquad \left.+ |\partial_{\z} v_i(t,\z)||\partial_{\z}b^N_{\epsilon,ij}(t,\z) | \right) p(\z) \,\mathrm{d} \z \\
& = 2 C_A \sum_{j=1}^N w_j \left(I + II \right).
\end{align}
Applying Young's inequality on the term $I$, we find
\begin{equation}
I \leq \| \partial_{\z} v_i(t,\z) \|^2_{L^2(I_\z)} + C_B,
\end{equation}
since 
\begin{equation}
\int_{I_{\z}} |b^N_{\epsilon,ij}(t,\z)|^2 p(\z) \,\mathrm{d} \z \leq 4\frac{\log^2 N}{\epsilon^2} \left(\frac{2\epsilon}{\pi}\right)^{d} \coloneqq C_B .
\end{equation}
To control term $II$, we consider the symmetric regularisation and we exploit the finiteness of the zeroth and second moments of the Gaussian mollifier:
\begin{align}
II &\leq \log(N) \int_{I_{\z}} |\partial_{\z} v_i(t,\z)| \left(\int_{\R^{d}} |\partial_{\z}\nabla\left( \psi_\epsilon(v_i-v) - \psi_\epsilon(v_j-v) \right) |\,\mathrm{d}v \right)p(\z) \,\mathrm{d} \z \\
& \leq C_{\partial_{\z}D} \int_{I_{\z}} |\partial_{\z} v_i(t,\z)| \left( |\partial_{\z} v_i(t,\z)| + |\partial_{\z} v_j(t,\z)| \right)p(\z) \,\mathrm{d} \z \\
& \leq C_{\partial_{\z}D} \left( \| \partial_{\z} v_i(t,\z) \|^2_{L^2(I_\z)} + \| \partial_{\z} v_i(t,\z) \|_{L^2(I_\z)} \| \partial_{\z} v_j(t,\z) \|_{L^2(I_\z)} \right).
\end{align}

We thus have
\begin{align}
\frac{\mathrm{d}}{\mathrm{d}t} \| \partial_{\z} v_i(t,\z) \|^2_{L^2(I_{\z})} & \leq 2 C_A \sum_{j=1}^N w_j \bigg[ C_{\partial_{\z}D}\| \partial_{\z} v_i(t,\z) \|^2_{L^2(I_{\z})} \\
& \quad + \| \partial_{\z} v_i(t,\z) \|_{L^2(I_{\z})} \left(C_B + C_{\partial_{\z}D}\| \partial_{\z} v_j(t,\z) \|_{L^2(I_\z)}\right) \bigg].
\end{align}
Summing over $i$, we obtain
\begin{align}
\frac{\mathrm{d}}{\mathrm{d}t} \sum_{i=1}^{N} \| \partial_{\z} v_i(t,\z) \|^2_{L^2(I_\z)} & \leq \sum_{i=1}^{N} 2 C_A (1+ 3C_{\partial_{\z}D}) \| \partial_{\z} v_i(t,\z) \|^2_{L^2(I_\z)} + 2 N C_A C_B \\
& = \sum_{i=1}^{N} K_1 \| \partial_{\z} v_i(t,\z) \|^2_{L^2(I_\z)} + K_2.
\end{align}
Applying Gronwall's lemma, we finally arrive at
\begin{equation}
\sum_{i=1}^{N} \| \partial_{\z} v_i(t,\z) \|^2_{L^2(I_\z)} \leq e^{ K_1 t} \sum_{i=1}^{N} \| \partial_{\z} v_i(0,\z) \|^2_{L^2(I_\z)} + K_2 t.
\end{equation}
\end{proof}

\begin{remark}
Theorems \ref{th:vz}-\ref{th:dzvz} imply that, provided $v_i(0,\z)$ and $\partial_{\z} v_i(0,\z)$ are in $L^2(I_\z)$ at the initial time $t=0$, then $v_i(t,\z)$ and $\partial_{\z} v_i(t,\z)$ remain in $L^2(I_\z)$ for times $t\in[0,T_f]$, with $T_f<+\infty$, under suitable assumptions. Therefore, in the time span $[0,T_f]$, we have 
\begin{equation}
v_i(t,\z)\in H^1(I_\z)\coloneqq\left\{ v \; : \; I_\z \to \R^d \; | \; \frac{\partial^k v}{\partial \z^k} \in L^2(I_\z), \; k=0,1 \right\}.
\end{equation}
\end{remark}

\section{Numerical tests} \label{sec:num}

We now present numerical tests to validate the \acf{sg} particle method. In \textbf{Test 1}, we show spectral convergence for the Maxwell and Coulomb cases in the presence of an uncertain initial temperature. In \textbf{Test 2}, we investigate the case with uncertainty in the exponent $\gamma$, a mixed potential scenario with a deterministic equilibrium distribution. In \textbf{Test 3}, we consider a two-dimensional uncertain parameter, and we verify the spectral convergence of the method. In \textbf{Test 4}, we compare the \ac{sg} particle method against the exact \ac{bkw} solution of the Maxwell case with uncertainty. In \textbf{Test 5}, we test the method against Trubnikov's formula for both Maxwellian and Coulombian molecules. In \textbf{Test 6}, we study the trend to equilibrium of the Coulomb model from an initial condition consisting of the sum of three Gaussians.

We will discretise the time span in $n_{\textrm{TOT}}$ time steps of size $\Delta t>0$, and we shall denote with $v_i^{n}(\z)$ the approximation of $v_i(t^n,\z)$ at the discrete time $t^n$. We employ the forward Euler scheme for the time discretisation; \eqref{eq:unc_scheme} becomes
\begin{equation}
v_i^{n+1}(\z) = v_i^{n}(\z) - \Delta t \sum_{j=1}^N w_j A^n_{ij}(\z)\left(\nabla\frac{\delta \Ht^{N}_{\epsilon}}{\delta f}(v^n_i(\z)) - \nabla\frac{\delta \Ht^{N}_{\epsilon}}{\delta f}(v^n_j(\z)) \right),
\end{equation}
and consequently \eqref{eq:particleSG} reads
\begin{align} 
\hat{v}^{n+1}_{i,m} &= \hat{v}^{n}_{i,m} - \Delta t \int_{I_{\z}} \sum_{j=1}^N w_j A^{n,M}_{ij}(\z) b^{n,N,M}_{\epsilon,ij}(\z) \Psi_m(\z) p(\z) \,\mathrm{d} \z, \\
b^{n,N,M}_{\epsilon,ij}(\z) &= \nabla\frac{\delta \Ht^{N}_{\epsilon}}{\delta f}(v^{n,M}_i(\z)) - \nabla\frac{\delta \Ht^{N}_{\epsilon}}{\delta f}(v^{n,M}_j(\z)).
\end{align}

All the tests are obtained for the two dimensional velocity domain case $d=2$, with a fixed time step $\Delta t = 0.01$, and for a one or two dimensional uncertain parameter $d_{\z}=1,2$. In the case $\z=(z_1,z_2)$, we assume that the two components are statistically independent, such that $p(\z)=p(z_1)p(z_2)$. Following \cite{Jingwei2020}, the variance of the Gaussian mollified is always $\epsilon = h^2 $, where $h = 2L_v/N$ is the size of the discretisation of the truncated velocity domain $[-L_v,L_v]^2$, and where $N$ is the number of particles. The weights of the particles are fixed to $w_i=1/N$ for every $i=1,\dots,N$. We choose a collision strength $C=1/16$ in the definition of the operator $A$, except in \textbf{Test 5}, where we declare it case by case. The initialisation of the particles is performed with a Monte Carlo sampling; for further details, we refer to Appendix A of \cite{medaglia2022JCP}.

The choice of the orthogonal polynomials is done following the Wiener-Askey scheme \cite{Xiu2002,Xiu2010}. In particular, we will consider the parameter $\z$ under a Uniform or Beta distribution, respectively leading to the Legendre and Jacobi polynomial bases.

\subsection{Test 1: spectral convergence with uncertain temperature}

\begin{figure}
\centering
\includegraphics[width = 0.3\linewidth]{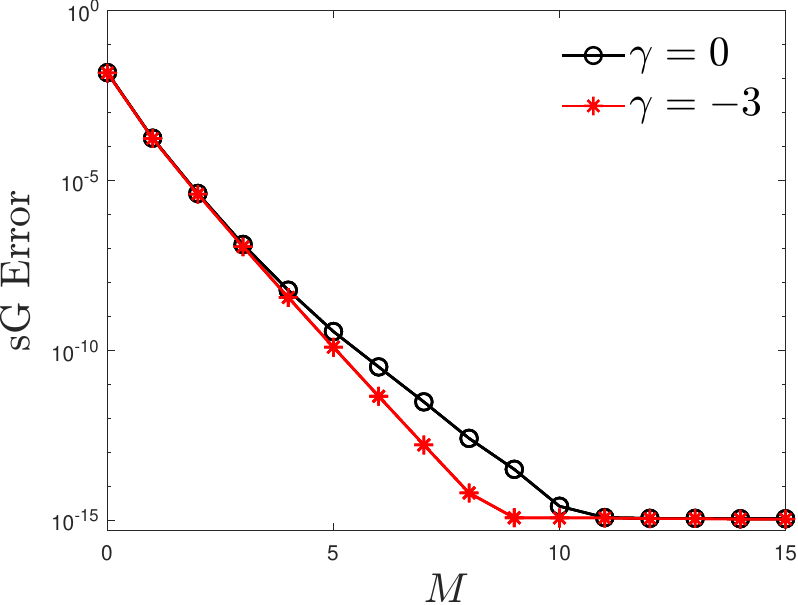}
\includegraphics[width = 0.3\linewidth]{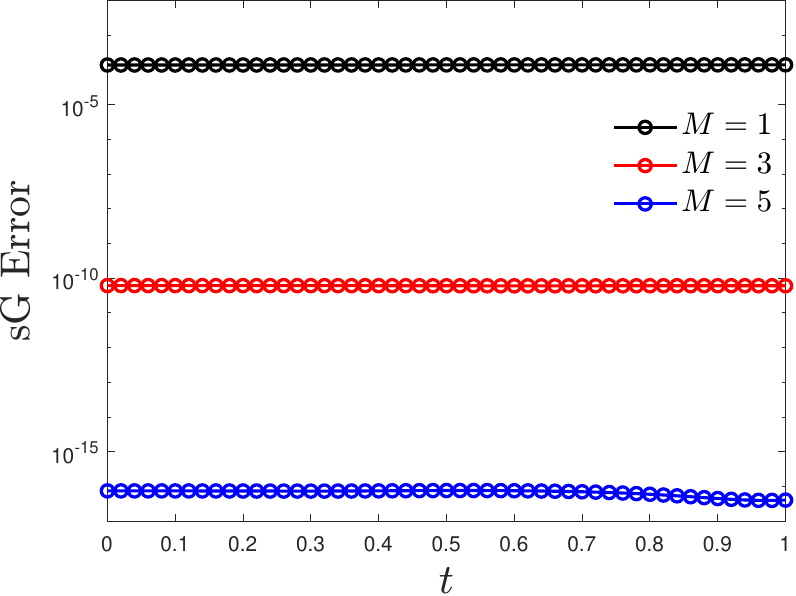}
\includegraphics[width = 0.3\linewidth]{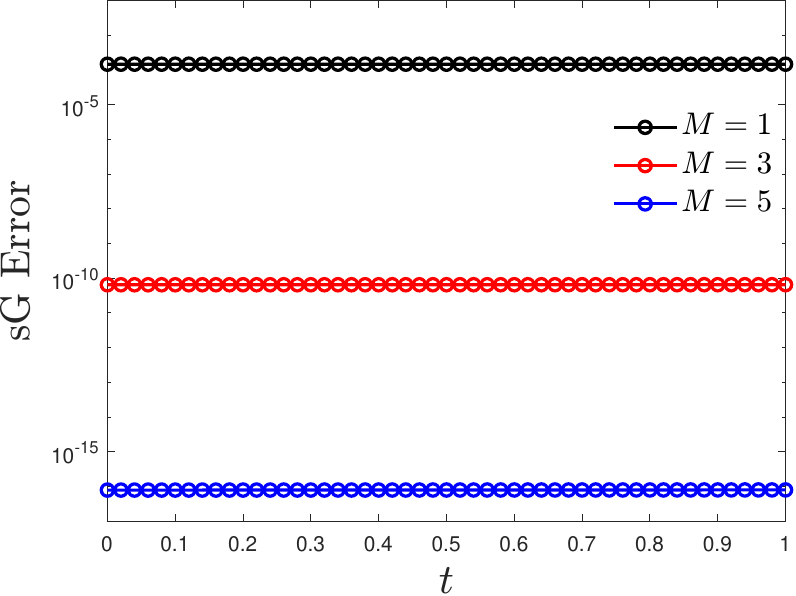}
\caption{\small{\textbf{Test 1}. Left: $L^2$ error of the fourth moment at time $t=1$, for increasing $M$, with respect to a reference solution, for the Maxwell ($\gamma=0$) and Coulomb ($\gamma=-3$) cases. Centre: time evolution of the same error for $\gamma=0$ in the time span $[0,1]$. Right: time evolution of the error for $\gamma=-3$. In all cases, $N=50^2$ and the reference solution is computed with $M^{\textrm{ref}}=30$. Initial conditions given by \eqref{eq:init_test1}}}
\label{fig:test_1}
\end{figure}

In this test, we consider the initial condition
\begin{equation} \label{eq:init_test1}
f^0(v,\z) = \frac{1}{\pi T^2(\z)} |v|^2 e^{-\frac{|v|^2}{T(\z)}} ,
\end{equation}
with $d_\z=1$, and uncertain initial temperature 
\begin{equation}
T(\z) = 1 + \frac{\z}{5},
\quad\textrm{with}\quad
\z\sim\mathcal{U}([0,1]).
\end{equation}
We choose the number of particles as $N=50^2$, and we compute a reference solution with $M^{\textrm{ref}}=30$, storing the initial particle sampling. Then, for different $M$, using the same sampling, we compute the $L^2$ error of the fourth moment of the solution:
\begin{equation}
\textrm{sG Error} = \| \textrm{M4}^{\textrm{ref}}(t,\z) -\textrm{M4}(t,\z) \|_{L^2(I_{\z})},
\end{equation}
where the moment is approximated by
\begin{equation}
\textrm{M4}(t,\z) \approx \frac{1}{N}\sum_{i=1}^{N} |v^M_i(t,\z)|^4. 
\end{equation} 

Figure \ref{fig:test_1} shows the \ac{sg} error for increasing values of $M$ (left), for both the Maxwell ($\gamma=0$) and Coulomb ($\gamma=-3$) cases at a fixed time $t=1$, as well as the time evolution for fixed $M=1,3,5$, both for $\gamma=0$ (centre) and $\gamma=-3$ (right). We observe that machine precision is reached spectrally fast, with a finite number of modes. Besides, the \ac{sg} Error proves approximately constant in time.

\subsection{Test 2: spectral convergence with uncertain potential}
\begin{figure}
\centering
\includegraphics[width = 0.3\linewidth]{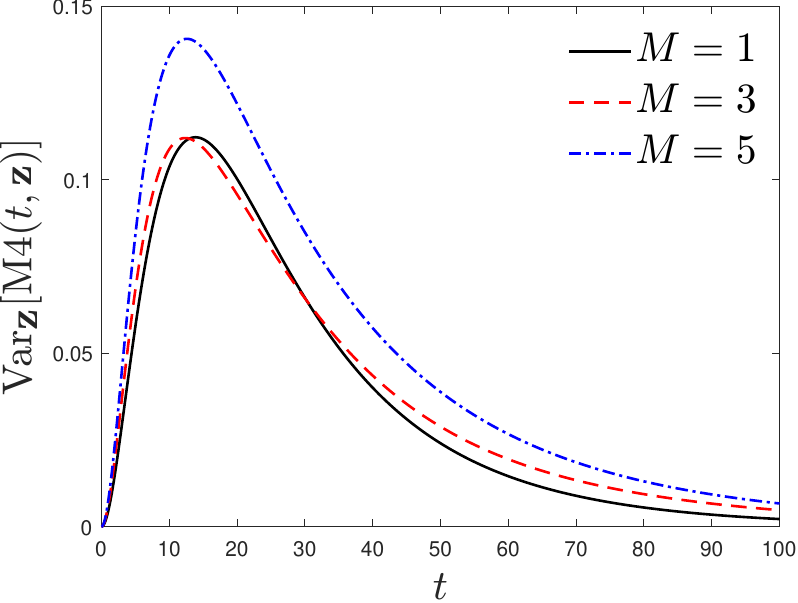}
\includegraphics[width = 0.3\linewidth]{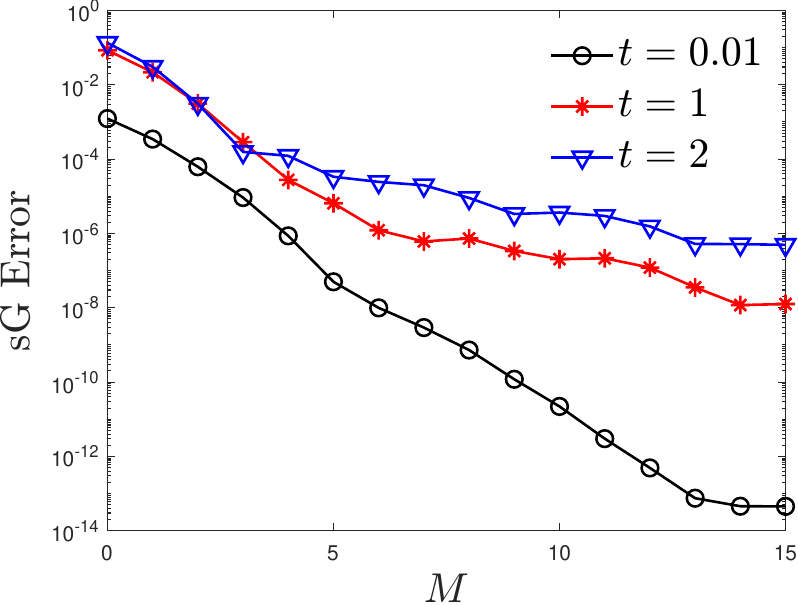}
\includegraphics[width = 0.3\linewidth]{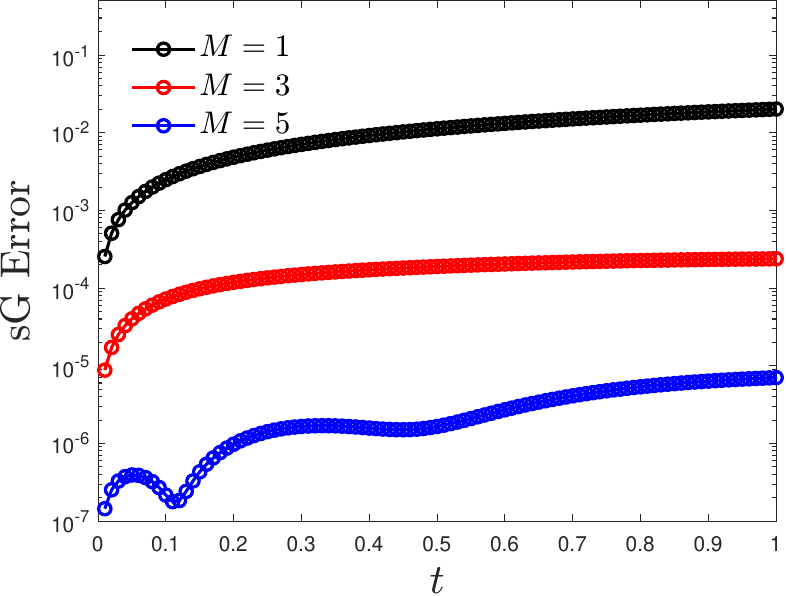}
\caption{\small{\textbf{Test 2}. Left: time evolution of the variance with respect to $\z$ of the fourth moment $\textrm{M4}(t,\z)$, for different values of $M$. Centre: $L^2$ error of $\textrm{M4}(t,\z)$ with respect to a reference solution, at times $t=0.01, 1, 2$, and increasing $M$. Right: time evolution in $[0,1]$ of the same error for fixed values of $M=1,3,5$. In all cases, $N=50^2$ and the reference solution is computed with $M^{\textrm{ref}}=30$. Initial conditions given by \eqref{eq:init_test1} with deterministic temperature $T=1$. The uncertainty is in $\gamma(\z) = - 3 \z$, with $\z\sim\mathcal{U}([0,1])$.}}
\label{fig:test_2}
\end{figure}
\begin{figure}
\centering
\includegraphics[width = 0.3\linewidth]{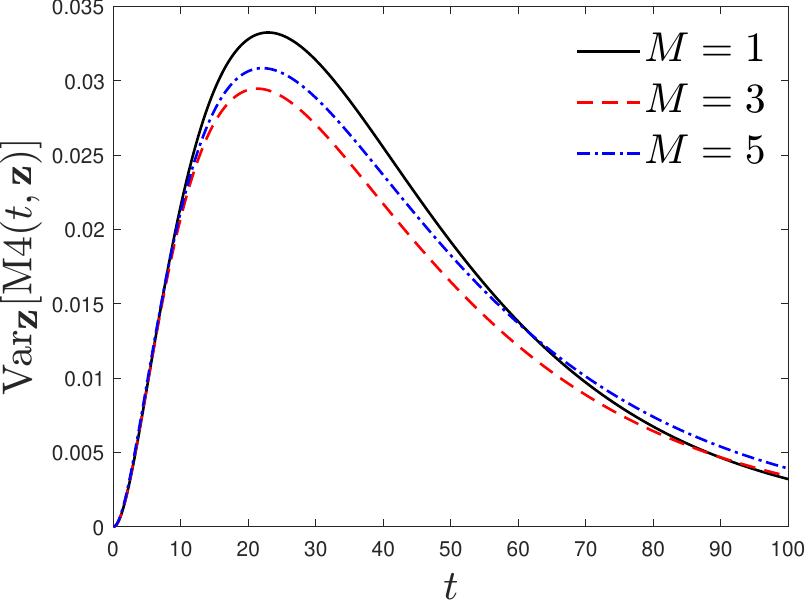}
\includegraphics[width = 0.3\linewidth]{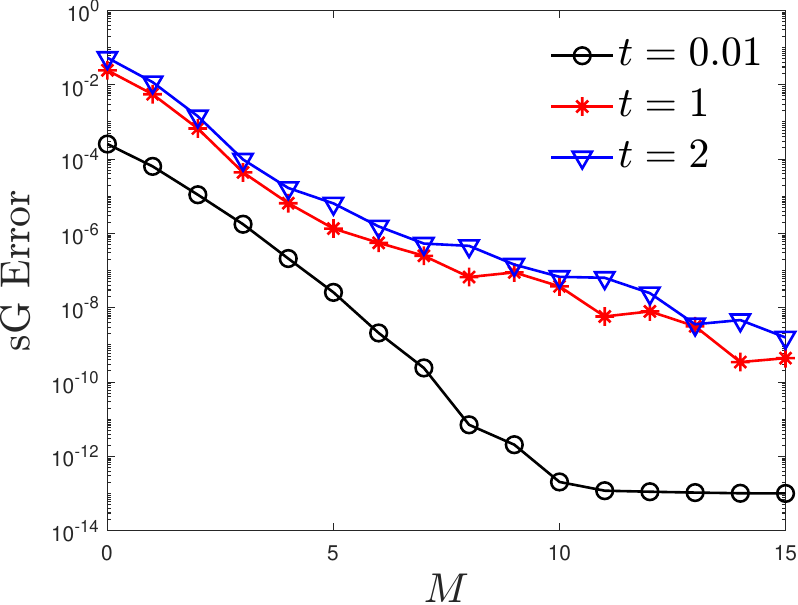}
\includegraphics[width = 0.3\linewidth]{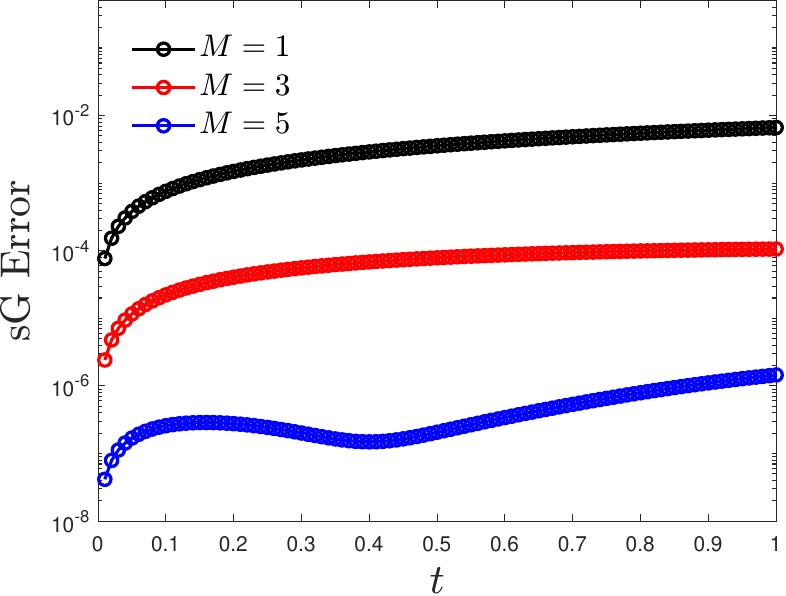}
\caption{\small{\textbf{Test 2}. Left: time evolution of the variance with respect to $\z$ of the fourth moment $\textrm{M4}(t,\z)$, for different values of $M$. Centre: $L^2$ error of $\textrm{M4}(t,\z)$ with respect to a reference solution, at times $t=0.01, 1, 2$, and increasing $M$. Right: time evolution in $[0,1]$ of the same error for fixed values of $M=1,3,5$. In all cases, $N=50^2$ and the reference solution is computed with $M^{\textrm{ref}}=30$. Initial conditions given by \eqref{eq:init_test1} with deterministic temperature $T=1$. The uncertainty is in $\gamma(\z) = - 3 \z$, with $\z\sim\mathrm{Beta}(2,5)$.}}
\label{fig:test_2_beta}
\end{figure}
\begin{figure}
\centering
\includegraphics[width = 0.3\linewidth]{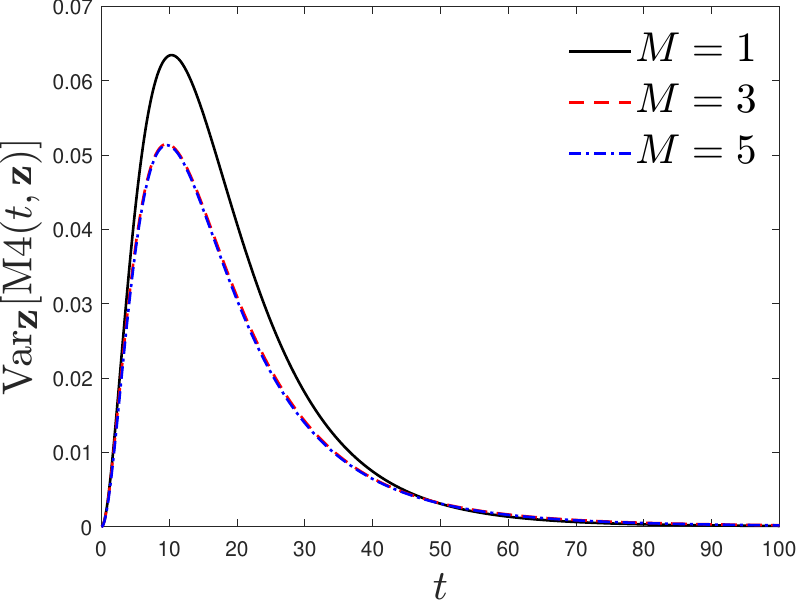}
\includegraphics[width = 0.3\linewidth]{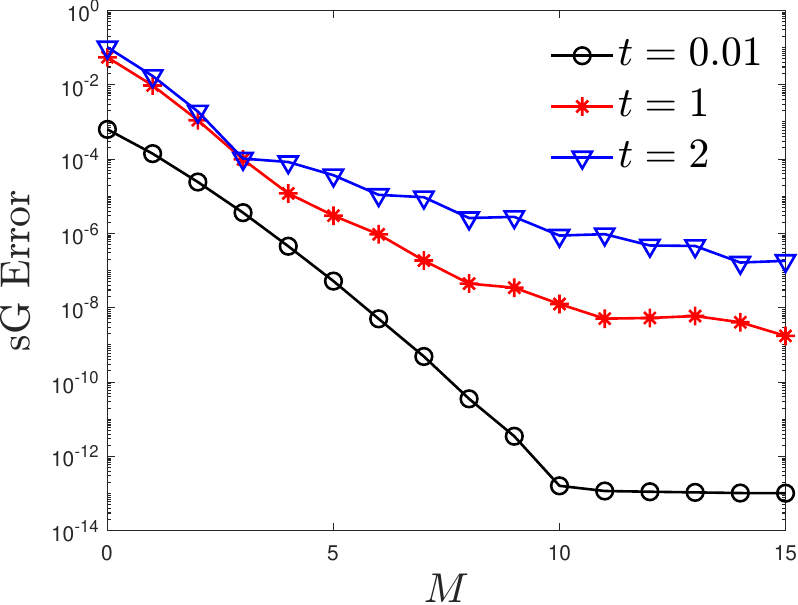}
\includegraphics[width = 0.3\linewidth]{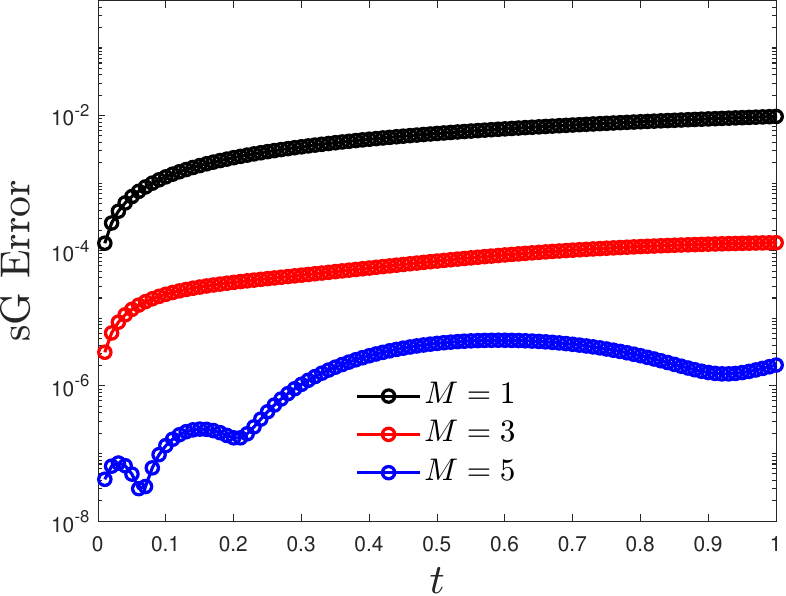}
\caption{\small{\textbf{Test 2}. Left: time evolution of the variance with respect to $\z$ of the fourth moment $\textrm{M4}(t,\z)$, for different values of $M$. Centre: $L^2$ error of $\textrm{M4}(t,\z)$ with respect to a reference solution, at times $t=0.01, 1, 2$, and increasing $M$. Right: time evolution in $[0,1]$ of the same error for fixed values of $M=1,3,5$. In all cases, $N=50^2$ and the reference solution is computed with $M^{\textrm{ref}}=30$. Initial conditions given by \eqref{eq:init_test1} with deterministic temperature $T=1$. The uncertainty is in $\gamma(\z) = - 3 \z$, with $\z\sim\mathrm{Beta}(5,2)$.}}
\label{fig:test_2_beta_v2}
\end{figure}

In this test, we consider a sample of $N=50^2$ particles of the initial conditions \eqref{eq:init_test1} with deterministic temperature $T=1$. The one dimensional uncertainty is now assumed in the exponent $\gamma$:
$
\gamma(\z) = - 3 \z.
$
We observe that, with these choices, the equilibrium distribution is deterministic; namely, the centred Gaussian with unit temperature. However, the dynamic is governed by an uncertain potential, ranging from the Coulombian to the Maxwell case. This is equivalent to considering the uncertainty in a relaxation parameter, as was already done in \cite{Dimarco22_2}, Test 1, case (a). 

In the left panel of Figure \ref{fig:test_2}, we show the variance with respect to uniformly distributed $\z\sim\mathcal{U}([0,1])$ of the time evolution of the fourth moment, for different $M$. In the central plot, we display the \ac{sg} error of $\textrm{M4}(t,\z)$ at times $t=0.01,1,2$, for increasing $M$, with respect to a reference solution computed with $M^{\textrm{ref}}=30$. Finally, in the right panel, we present the time evolution of the same error, for fixed $M=1,3,5$. Figures \ref{fig:test_2_beta}-\ref{fig:test_2_beta_v2} perform the same analysis for the distributions $\z\sim\mathrm{Beta}(2,5)$ and $\z\sim\mathrm{Beta}(5,2)$, respectively.

We note that the variance, which starts at zero, decreases in time after reaching a maximum. This is due to the fact that both the initial condition and the asymptotic equilibrium are deterministic, while the uncertainty is present only in the time evolution. Similarly to \cite{Dimarco22_2,gerritsma2010time}, the accuracy of the \ac{gpc} approximation deteriorates over time. This contrasts with Test 1, which had uncertainty in the initial temperature, and consequently also in the equilibrium state.

\subsection{Test 3: two dimensional uncertain parameter}
This tests combines Test 1 and Test 2. We consider now a two dimensional ($d_\z=2$) uncertain parameter $\z=(z_1,z_2)$ and we suppose that the components are statistically independent in a way that $p(\z)=p(z_1)p(z_2)$. We initialise the particles according to
\begin{equation} \label{eq:init_test3}
f^0(v,z_1) = \frac{1}{\pi T^2(z_1)} |v|^2 e^{-\frac{|v|^2}{T(z_1)}} 
\end{equation}
with uncertain initial temperature 
\begin{equation}
T(z_1) = 1 + \frac{z_1}{5},
\quad\textrm{with}\quad
z_1\sim\mathcal{U}([0,1]).
\end{equation}
Moreover, we assume that the exponent $\gamma$ depends on $z_2$:
$\gamma(z_2) = -3 z_2,$ where $z_2$ follows a uniform or Beta distribution. Under this hypothesis, the \ac{gpc} expansion \eqref{eq:gPC} becomes
\begin{equation}
v_i(t,z_1,z_2) \approx v^{M_1,M_2}_i(t,z_1,z_2) = \sum_{m=0}^{M_1} \sum_{n=0}^{M_2} \hat{v}_{i,m,n}(t) \Psi^{(1)}_m(z_1) \Psi^{(2)}_n(z_2)
\end{equation}
where $\{ \Psi^{(1)}_m(z_1) \}_{m=0}^{M_1}$ and $\{ \Psi^{(2)}_n(z_2) \}_{n=0}^{M_2}$ are the polynomials orthogonal with respect to $p(z_1)$ and $p(z_2)$, respectively.

Figures \ref{fig:test_2dz_unif_unif}-\ref{fig:test_2dz_unif_beta}-\ref{fig:test_2dz_unif_beta_v2} show the spectral error of the fourth moment $\textrm{M4}(t,z_1,z_2)$ with respect to a reference solution computed with $M^{\textrm{ref}}_1=M^{\textrm{ref}}_2=30$, for increasing order $M_1,M_2$, at times $t=0.01,1,2$. We consider $z_2\sim\mathcal{U}([0,1])$ in Figure \ref{fig:test_2dz_unif_unif}, $z_2\sim\mathrm{Beta}(2,5)$ in Figure \ref{fig:test_2dz_unif_beta}, and $z_2\sim\mathrm{Beta}(5,2)$ in Figure \ref{fig:test_2dz_unif_beta_v2}. The number of particles is $N=50^2$. The error is presented in $\log_{10}$ scale in all the figures.

The results are consistent with the previous tests. We note that, as the time increases, the error in $z_2$ (and thus the global error) deteriorates.

\begin{figure}
\centering
\includegraphics[width = 0.3\linewidth]{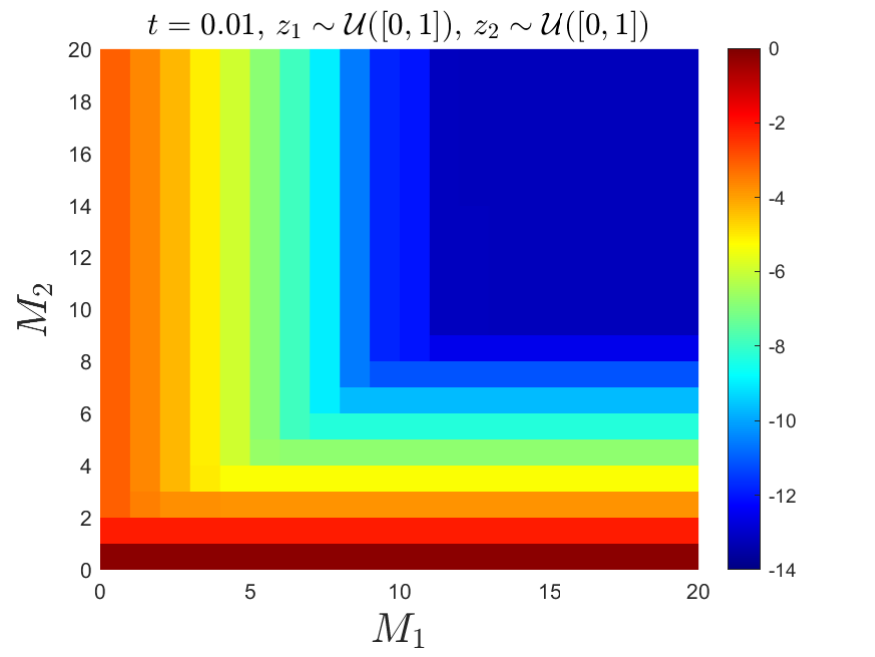}
\includegraphics[width = 0.3\linewidth]{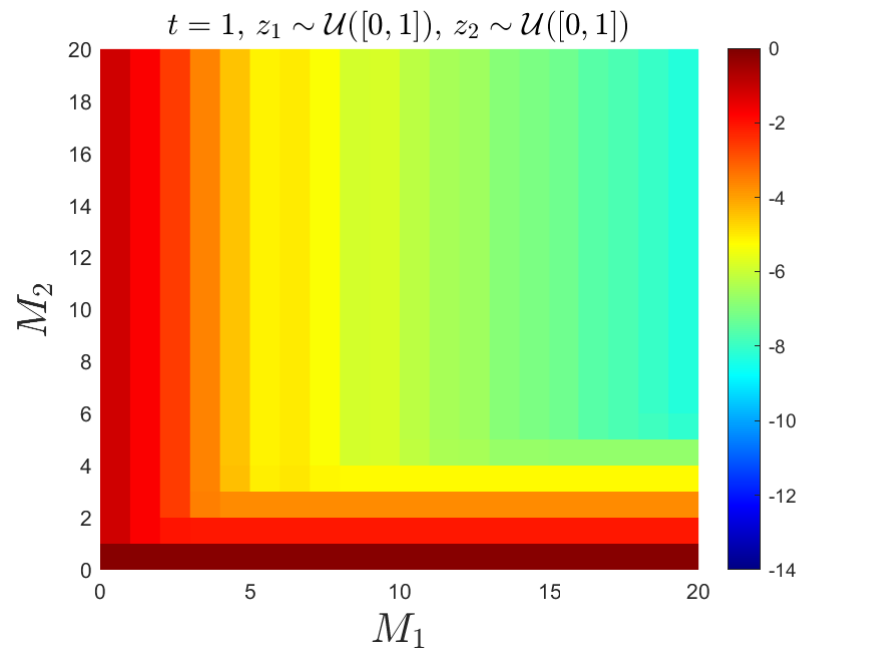}
\includegraphics[width = 0.3\linewidth]{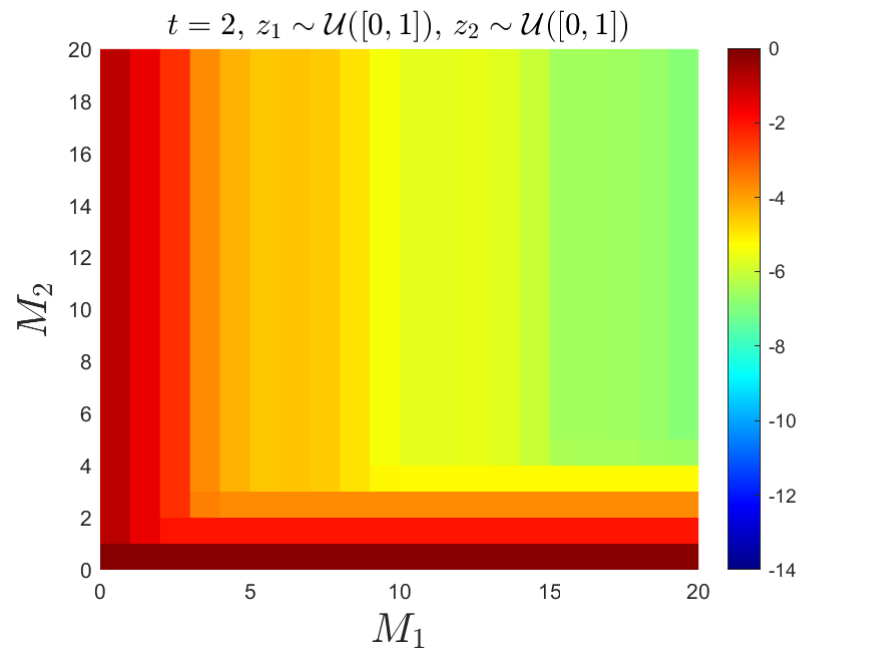}
\caption{\small{\textbf{Test 3}. Convergence of the $L^2$ error of fourth moment $\textrm{M4}(t,z_1,z_2)$ at times $t=0.01$ (left), $t=1$ (centre), $t=2$ (right), for increasing $M_1$ and $M_2$. In all cases, $N=50^2$ and $M^{\textrm{ref}}_1=M^{\textrm{ref}}_2=30$. Initial conditions given by \eqref{eq:init_test3} with $\gamma(z_2) = -3 z_2$, and $z_2\sim\mathcal{U}([0,1])$. The error is presented in $\log_{10}$ scale.}}
\label{fig:test_2dz_unif_unif}
\end{figure}
\begin{figure}
\centering
\includegraphics[width = 0.3\linewidth]{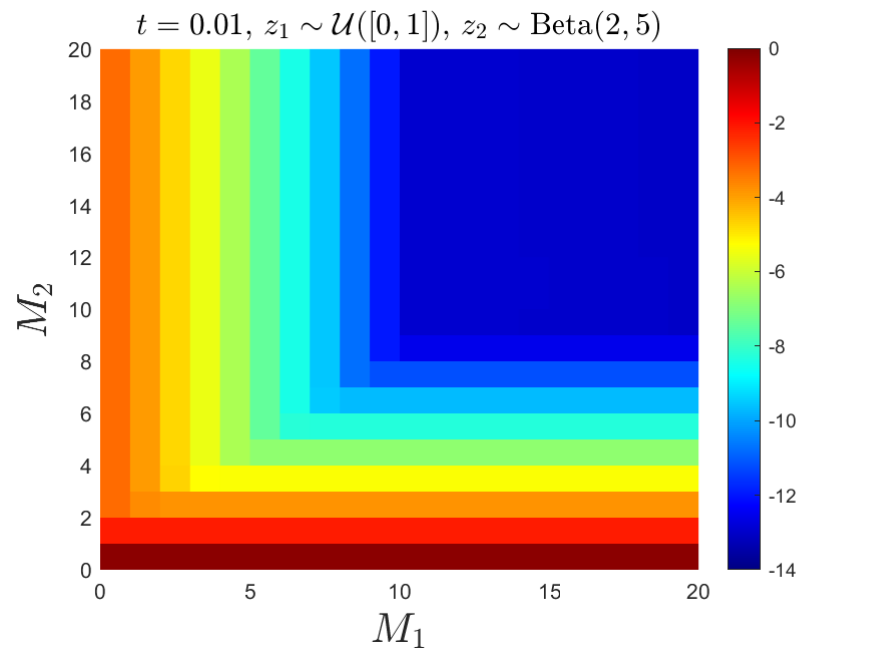}
\includegraphics[width = 0.3\linewidth]{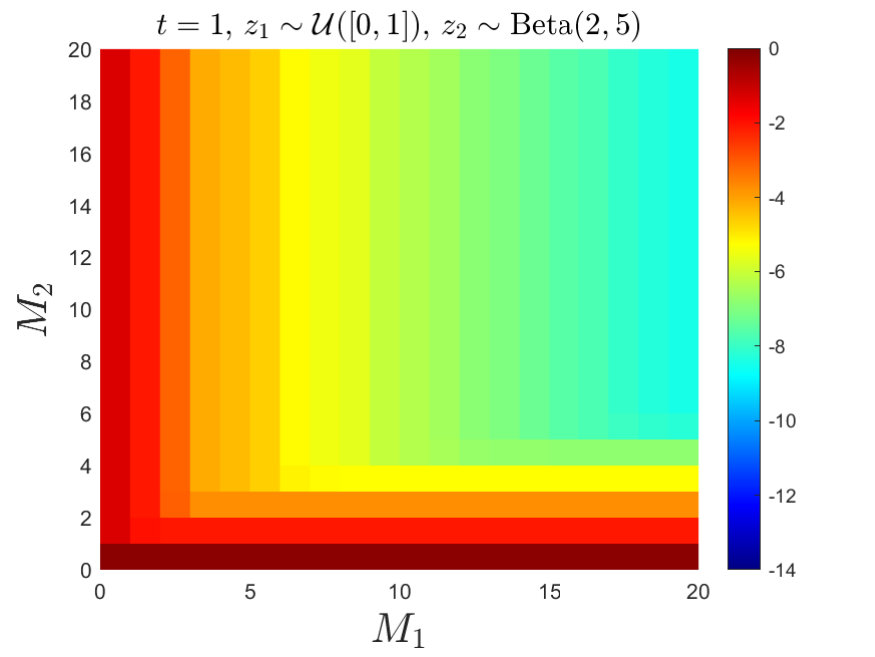}
\includegraphics[width = 0.3\linewidth]{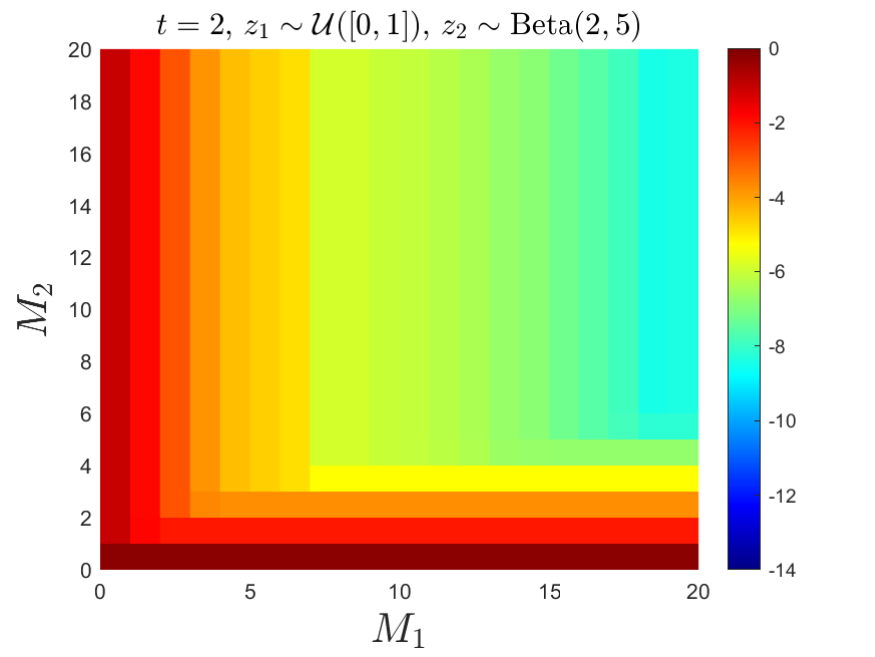}
\caption{\small{\textbf{Test 3}. Convergence of the $L^2$ error of the fourth moment $\textrm{M4}(t,z_1,z_2)$ at times $t=0.01$ (left), $t=1$ (centre), $t=2$ (right), for increasing $M_1$ and $M_2$. In all cases, $N=50^2$ and $M^{\textrm{ref}}_1=M^{\textrm{ref}}_2=30$. Initial conditions given by \eqref{eq:init_test3} with $\gamma(z_2) = -3 z_2$, and $z_2\sim\mathrm{Beta}(2,5)$. The error is presented in $\log_{10}$ scale.}}
\label{fig:test_2dz_unif_beta}
\end{figure}
\begin{figure}
\centering
\includegraphics[width = 0.3\linewidth]{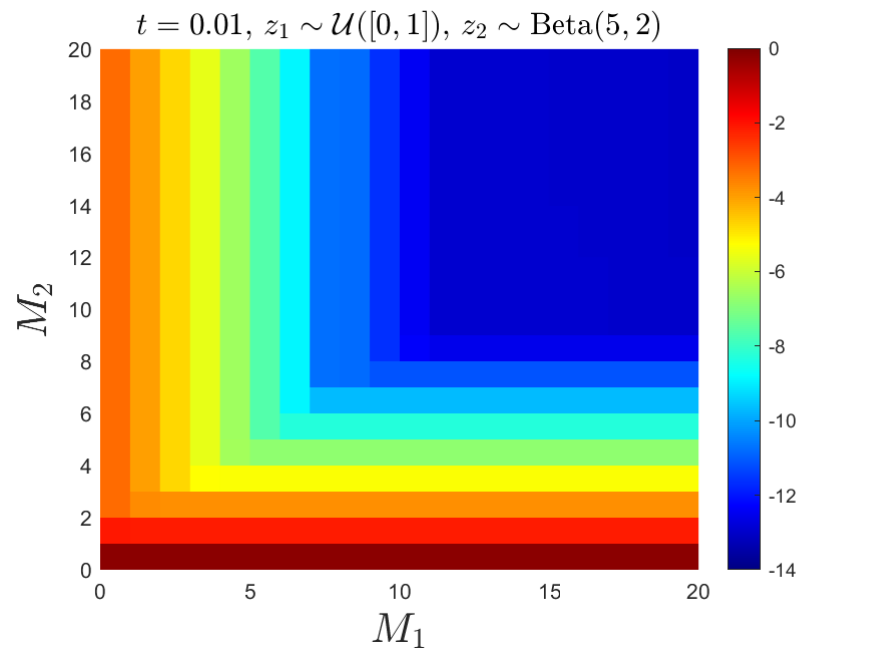}
\includegraphics[width = 0.3\linewidth]{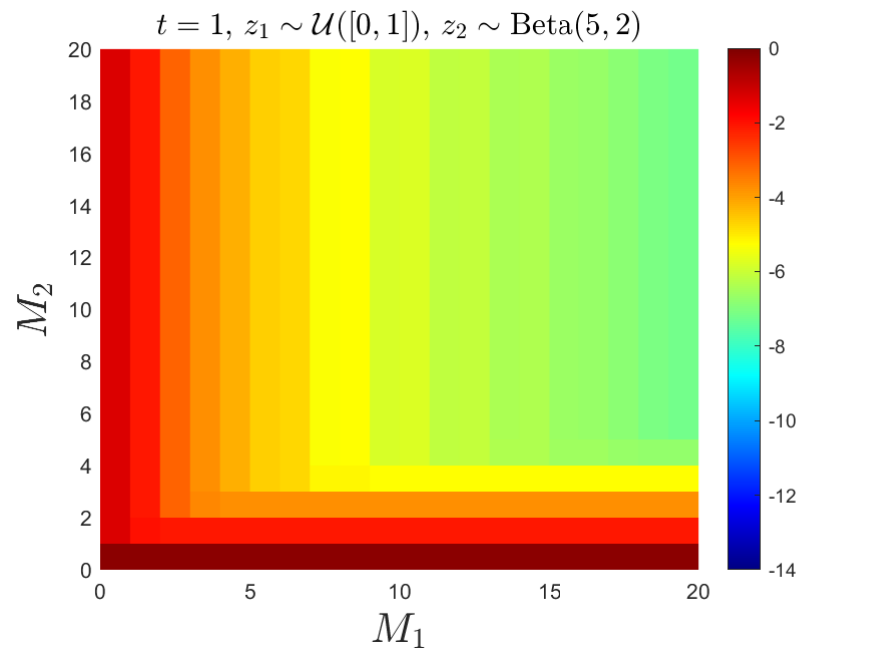}
\includegraphics[width = 0.3\linewidth]{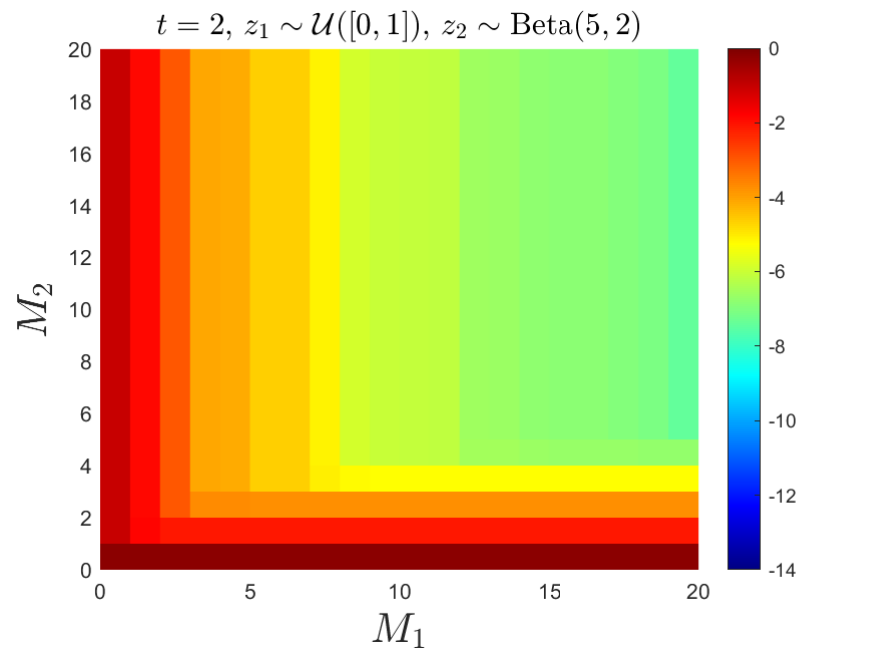}
\caption{\small{\textbf{Test 3}. Convergence of the $L^2$ error of the fourth moment $\textrm{M4}(t,z_1,z_2)$ at times $t=0.01$ (left), $t=1$ (centre), $t=2$ (right), for increasing $M_1$ and $M_2$. In all cases, $N=50^2$ and $M^{\textrm{ref}}_1=M^{\textrm{ref}}_2=30$. Initial conditions given by \eqref{eq:init_test3} with $\gamma(z_2) = -3 z_2$, and $z_2\sim\mathrm{Beta}(5,2)$. The error is presented in $\log_{10}$ scale.}}
\label{fig:test_2dz_unif_beta_v2}
\end{figure}

\subsection{Test 4: BKW solution}
\begin{figure}
\centering
\includegraphics[width = 0.3\linewidth]{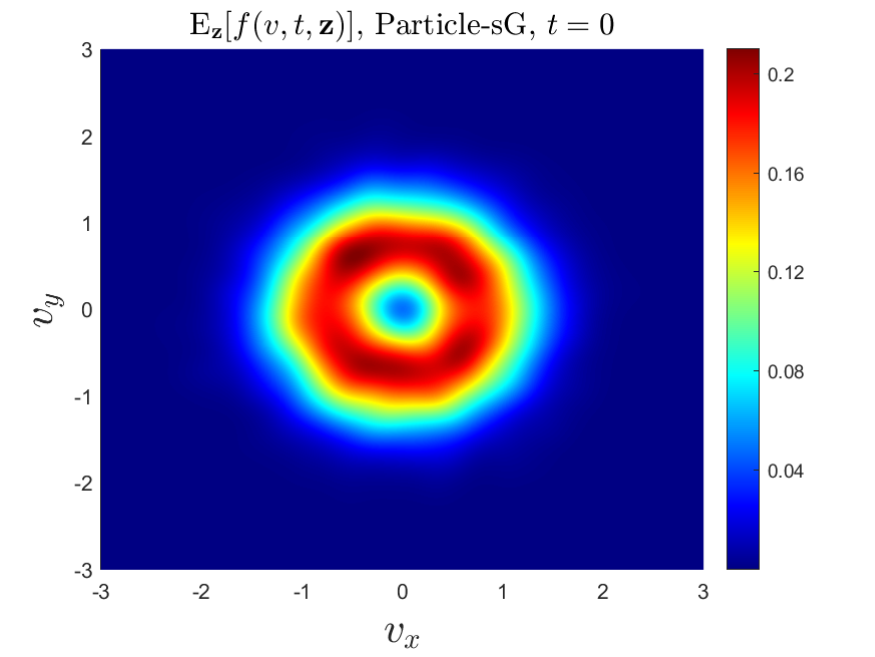}
\includegraphics[width = 0.3\linewidth]{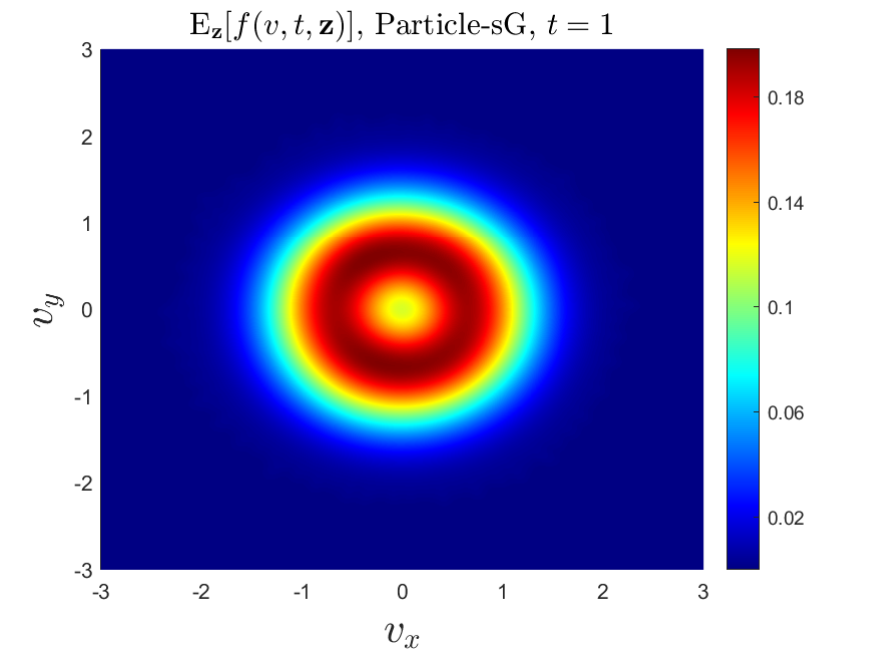}
\includegraphics[width = 0.3\linewidth]{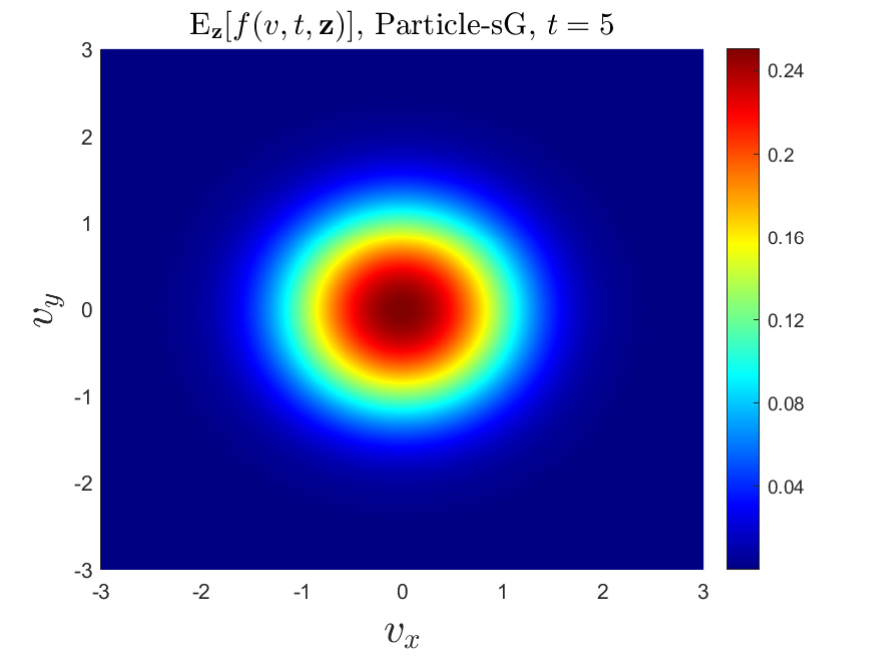}
\includegraphics[width = 0.3\linewidth]{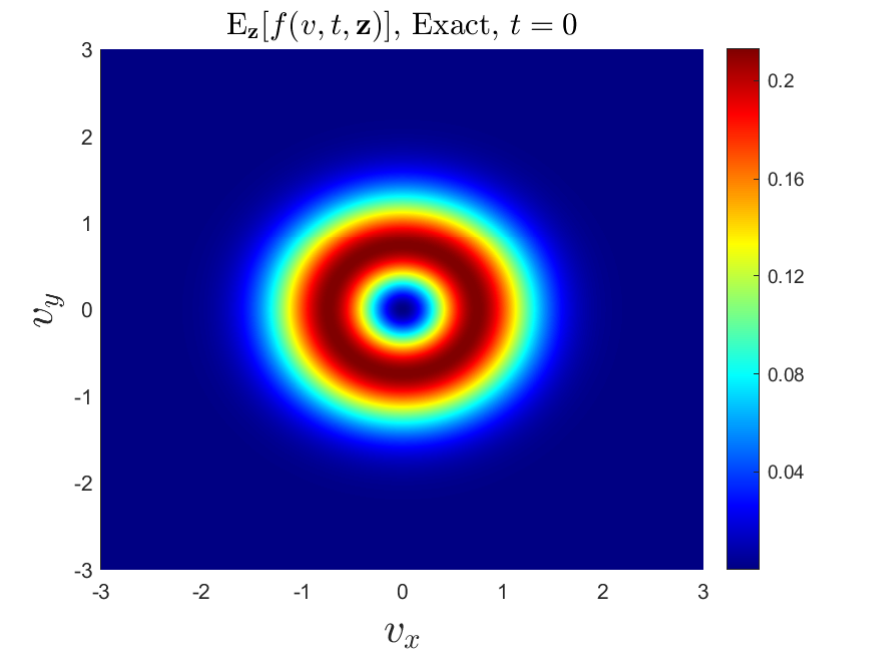}
\includegraphics[width = 0.3\linewidth]{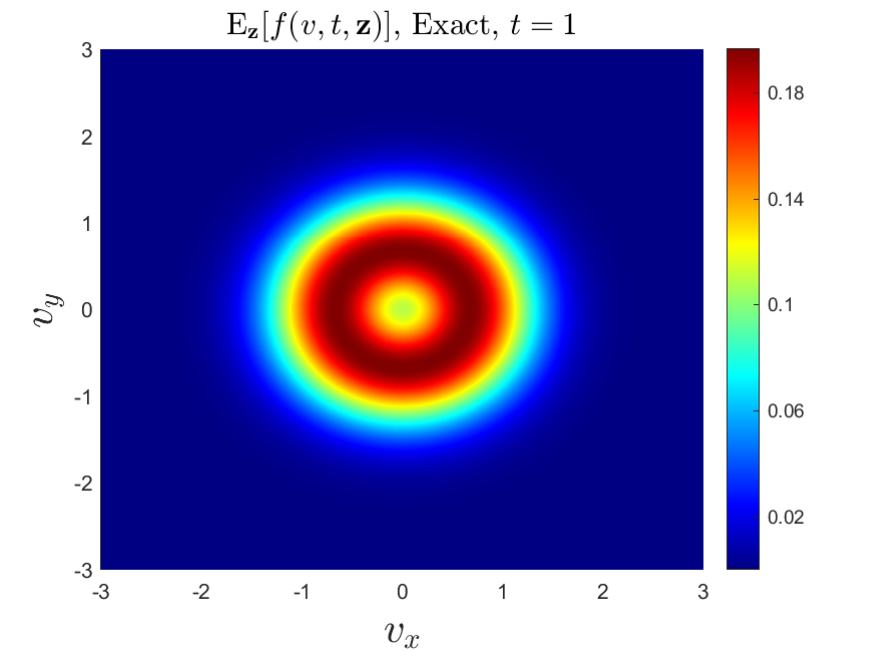}
\includegraphics[width = 0.3\linewidth]{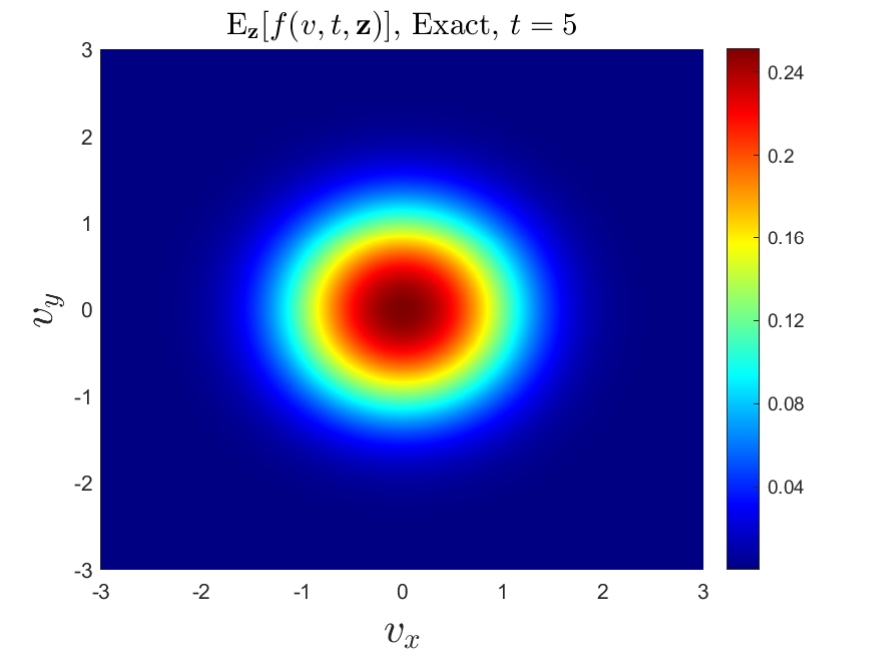}
\caption{\small{\textbf{Test 4}. Expected distributions $\mathbb{E}_{\z}[f(v,t,\z)]$ at times $t=0,1,5$ for the \ac{bkw} test. Upper row: \ac{sg} particle solution obtained with $N=120^2$ and $M=3$. Lower row: exact \ac{bkw} solution \eqref{eq:BKW}.}}
\label{fig:test_3_exp}
\end{figure}
\begin{figure}
\centering
\includegraphics[width = 0.3\linewidth]{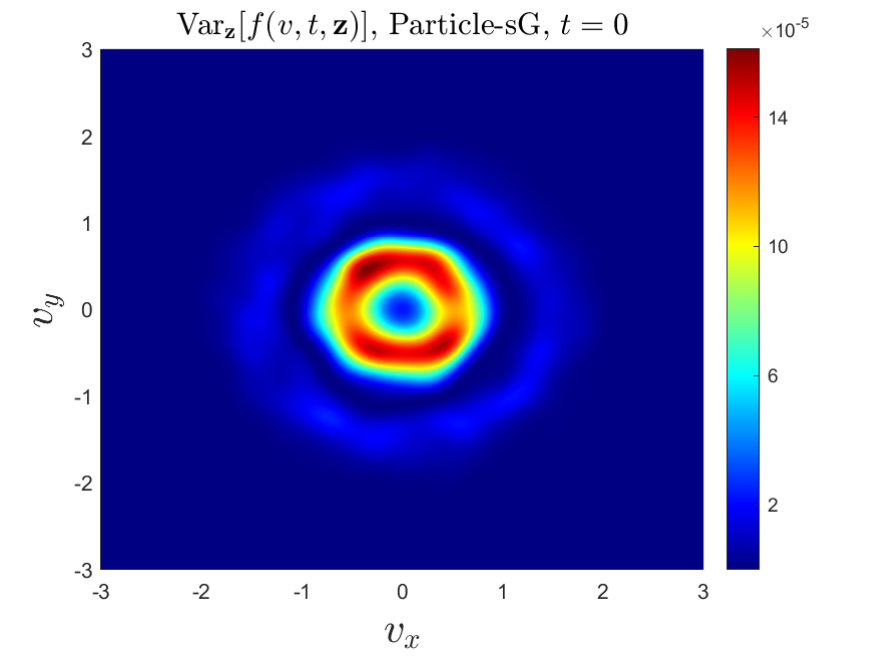}
\includegraphics[width = 0.3\linewidth]{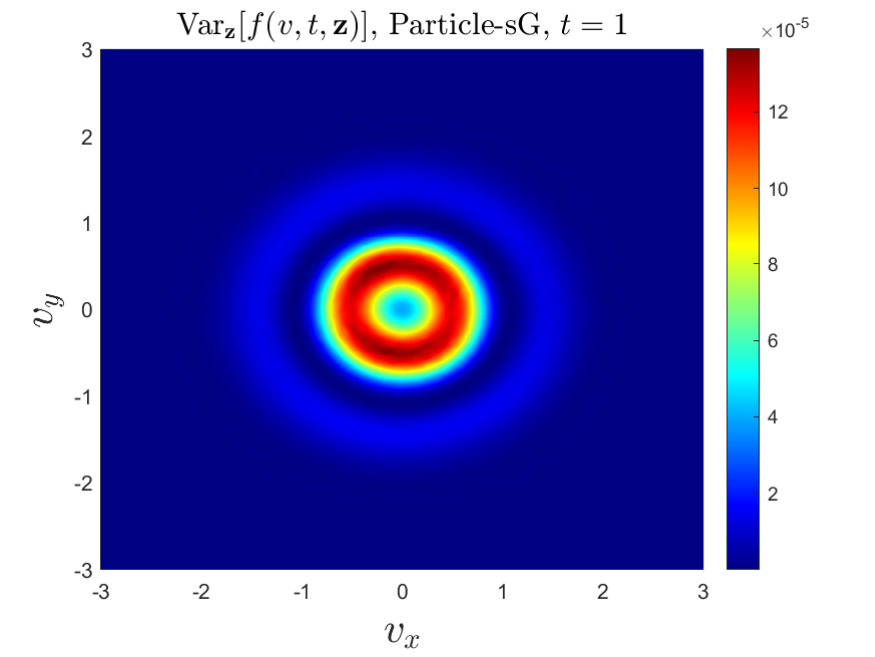}
\includegraphics[width = 0.3\linewidth]{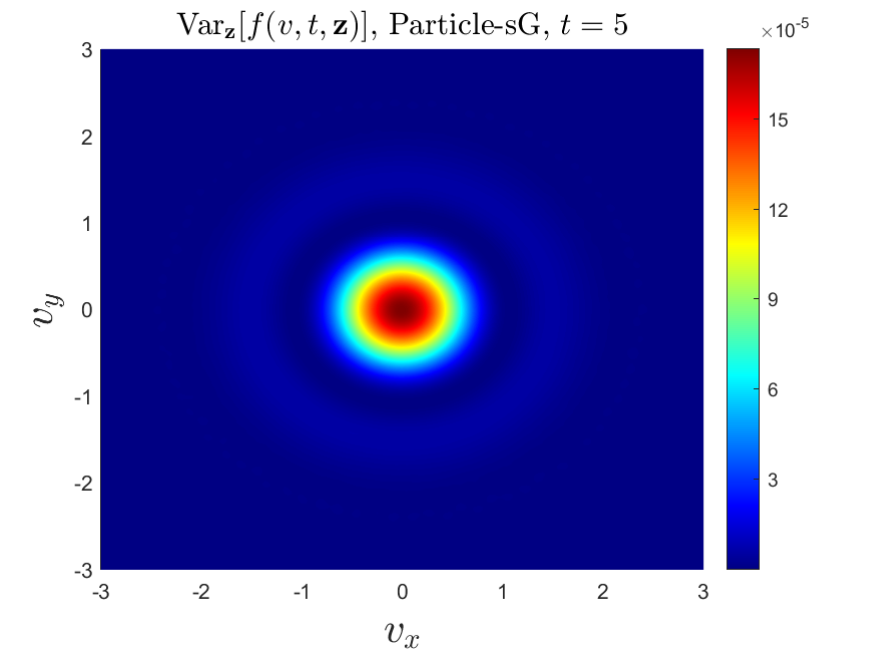}
\includegraphics[width = 0.3\linewidth]{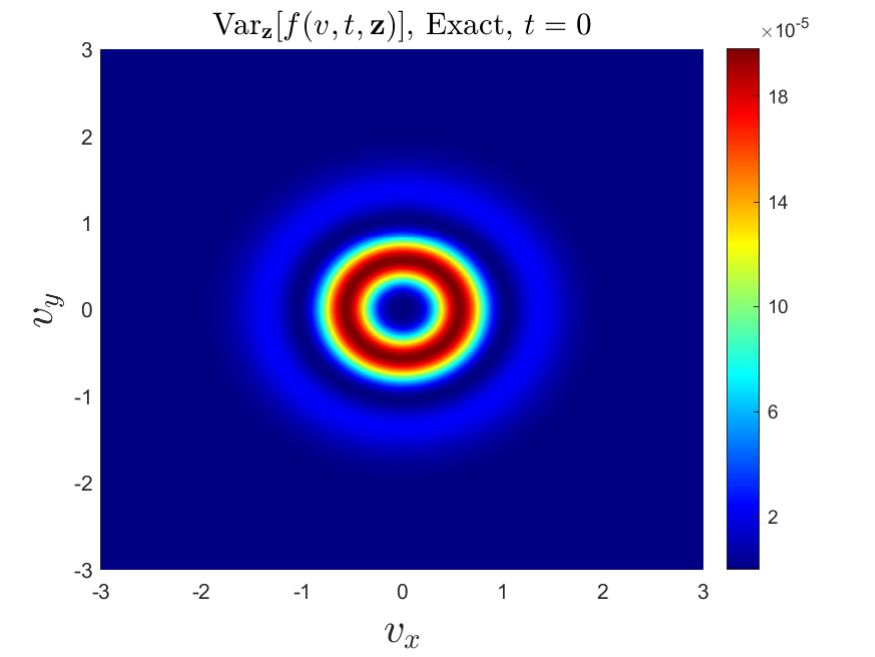}
\includegraphics[width = 0.3\linewidth]{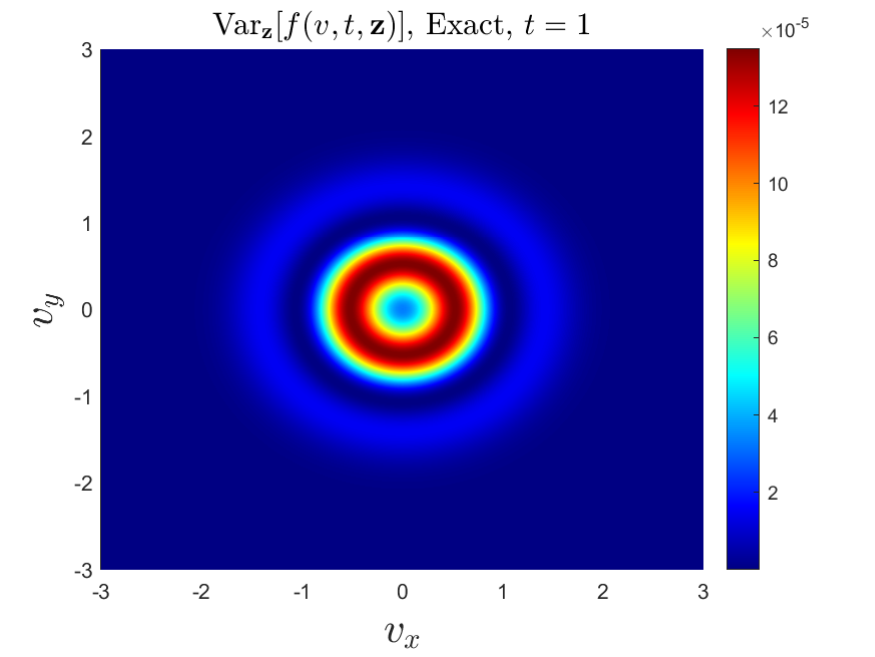}
\includegraphics[width = 0.3\linewidth]{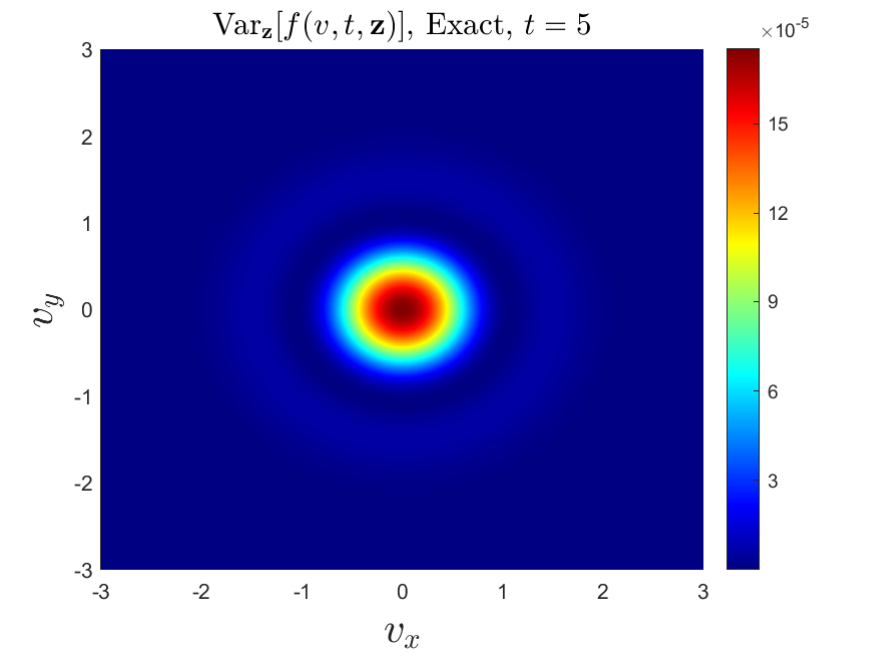}
\caption{\small{\textbf{Test 4}. Variance of the distributions $\textrm{Var}_{\z}[f(v,t,\z)]$ at times $t=0,1,5$ for the \ac{bkw} test. Upper row: \ac{sg} particle solution obtained with $N=120^2$ and $M=3$. Lower row: exact \ac{bkw} solution \eqref{eq:BKW}.}}
\label{fig:test_3_var}
\end{figure}
\begin{figure}
\centering
\includegraphics[width = 0.3\linewidth]{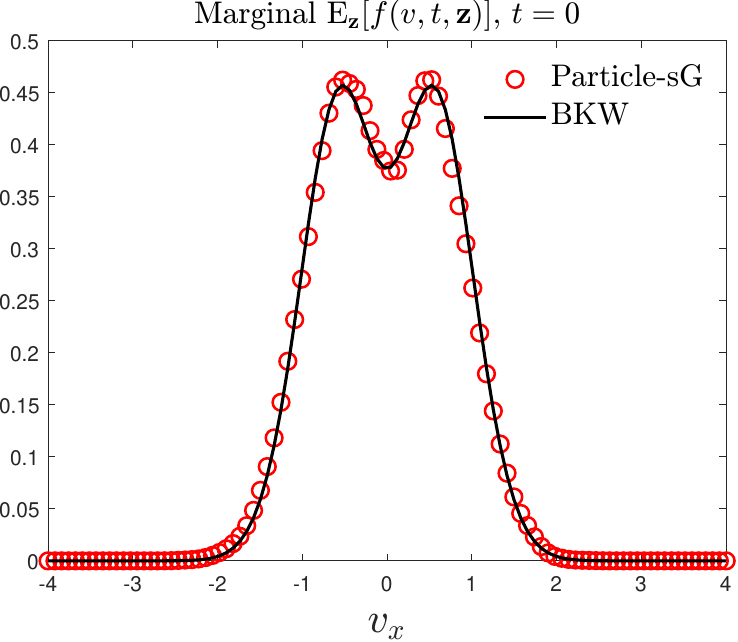}
\includegraphics[width = 0.3\linewidth]{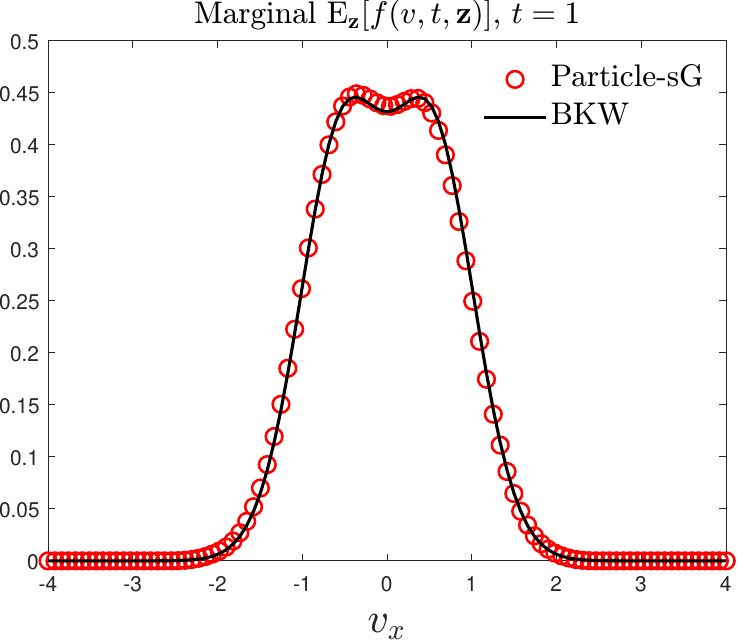}
\includegraphics[width = 0.3\linewidth]{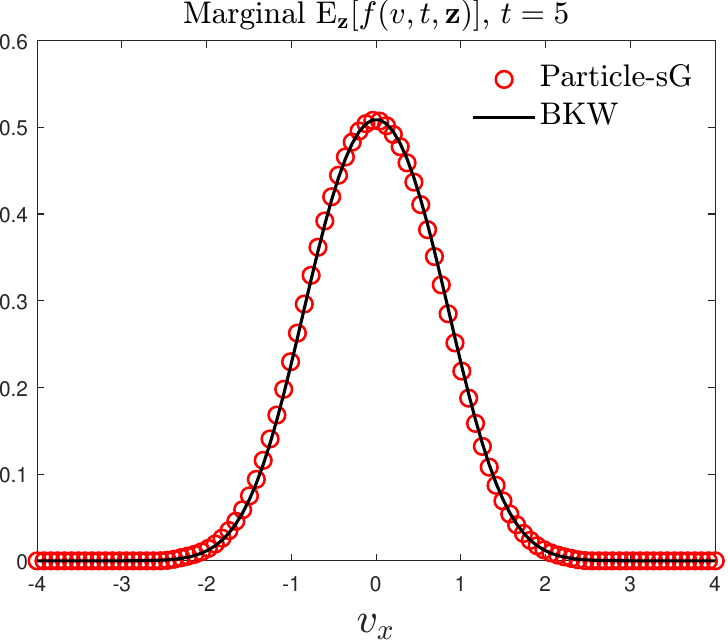}
\includegraphics[width = 0.3\linewidth]{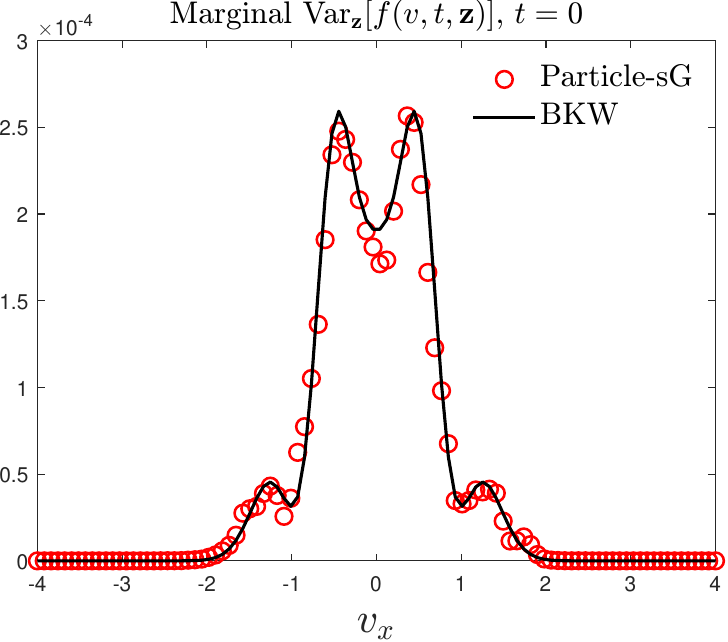}
\includegraphics[width = 0.3\linewidth]{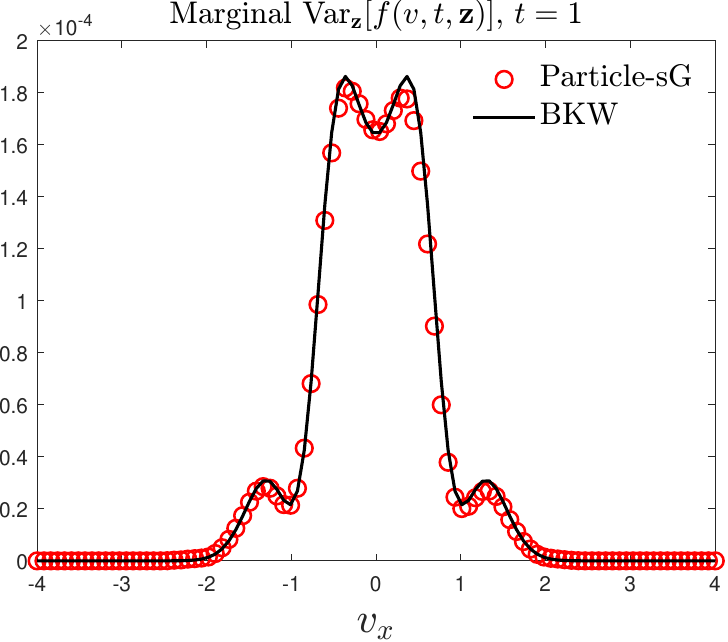}
\includegraphics[width = 0.3\linewidth]{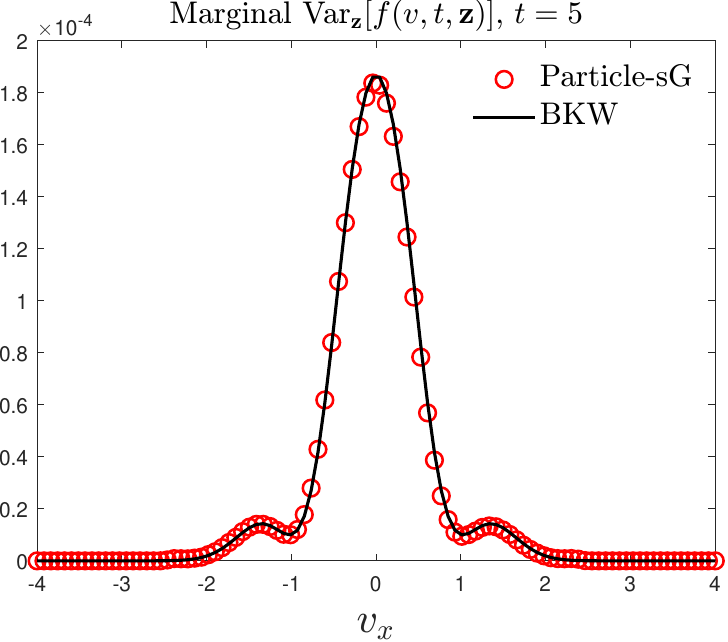}
\caption{\small{\textbf{Test 4}. Marginals of $\mathbb{E}_{\z}[f(v,t,\z)]$ and $\textrm{Var}_{\z}[f(v,t,\z)]$ at times $t=0,1,5$, of the \ac{sg} particle approximation and the exact \ac{bkw} solution \eqref{eq:BKW}. Numerical solution obtained with $N=120^2$ and $M=3$.}}
\label{fig:test_3_marginal}
\end{figure}
In this test we consider the model with Maxwellian molecules, i.e., $\gamma=0$, with $d_\z=1$ and uncertain initial temperature. In this scenario, a benchmark is given by the \acf{bkw} solution (see, for instance, Appendix A of the works \cite{Jingwei2020, Pareschi2020})
\begin{equation} \label{eq:BKW}
f(v,t,\z) = \frac{1}{2\pi K(t,\z)} e^{-\dfrac{|v|^2}{2K(t,\z)}} \left( \frac{2K(t,\z)-T(\z)}{K(t,\z)} + \frac{T(\z)-K(t,\z)}{2K^2(t,\z)} |v|^2 \right),
\end{equation}
where $K$ is given by
\begin{equation}
K(t,\z) = T(\z) \left(1 - \frac{1}{2} e^{-t/8}\right).
\end{equation}
We initialise the \ac{sg} particle scheme by sampling $N$ particles from \eqref{eq:BKW} at $t=0$ and uncertain initial temperature:
\begin{equation}
T(\z) = 0.5 + 0.1\z,
\quad\textrm{with}\quad
\z\sim\mathcal{U}([0,1]).
\end{equation}
We fix $M=3$ as the order of the \ac{gpc} expansion, and we compare the numerical solution with the exact solution. We choose $N=120^2$ for Figures \ref{fig:test_3_exp}-\ref{fig:test_3_var}-\ref{fig:test_3_marginal}, and $N=60^2,80^2,100^2,120^2$ for Figure \ref{fig:test_3_L2err}. 

Figures \ref{fig:test_3_exp}-\ref{fig:test_3_var}-\ref{fig:test_3_marginal} show the comparison between the exact \ac{bkw} solution \eqref{eq:BKW} and the \ac{sg} particle approximation. In the first two, we report at times $t=0,1,5$ the expectation $\mathbb{E}_{\z}[f(v,t,\z)]$ and the variance $\textrm{Var}_{\z}[f(v,t,\z)]$ of the distributions. In the third one, we compare directly the marginals of the expectation and the variance of the distributions. The \ac{sg} particle method matches the exact \ac{bkw} solution very well.

\begin{figure}
\centering
\includegraphics[width = 0.5\linewidth]{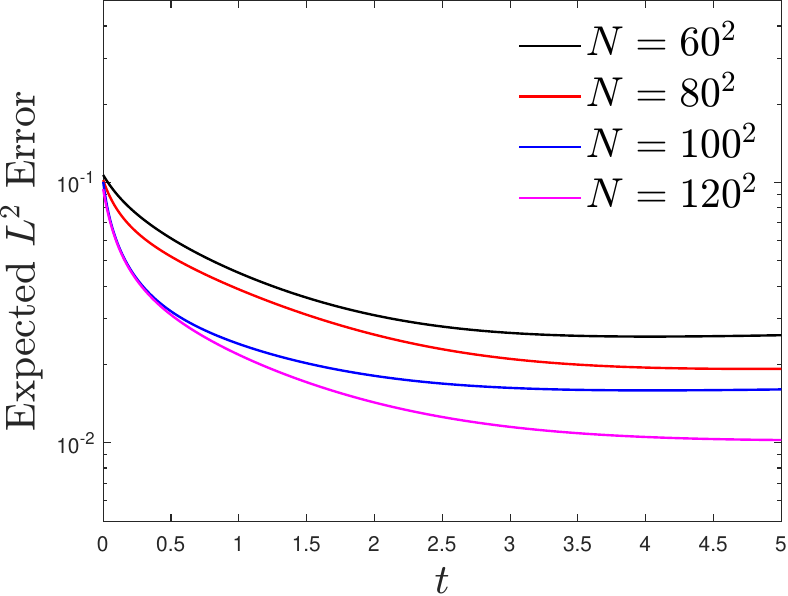}
\caption{\small{\textbf{Test 4}. Expectation in $\z$ of the relative $L^2$ error with respect to the exact \ac{bkw} solution, as a function of the time, for different number of particles $N=60^2,80^2,100^2,120^2$. In all the tests, we choose $M=3$.}}
\label{fig:test_3_L2err}
\end{figure}

In Figure \ref{fig:test_3_L2err}, the expectation in $\z$ of the relative $L^2$ error with respect to the exact \ac{bkw} solution \eqref{eq:BKW} is displayed:
\begin{equation}
L^2\, \textrm{Error} = \left( \int_{\R^{2}} \frac{|f^{\textrm{BKW}}(v,t,\z)-f^{\textrm{num}}(v,t,\z)|^2}{|f^{\textrm{BKW}}(v,t,\z)|^2} \,\mathrm{d}v \right)^{1/2}.
\end{equation}
The error decreases as the number of particles increases, and it also decreases monotonically in time. 

\subsection{Test 5: Trubnikov's formula}
We study here the Trubnikov's formula, which describes the relaxation towards equilibrium of anisotropic initial temperatures in the Maxwellian and Coulombian scenarios. The particles are initialised as
\begin{equation} \label{eq:init_trubnikov}
f^0(v,\z) = \frac{1}{2\pi} \frac{1}{\sqrt{T_x(\z) T_y}} \exp\left\{-\dfrac{v^2_x}{2 T^0_x(\z)}\right\} \exp\left\{-\dfrac{v^2_y}{2 T^0_y}\right\}
\end{equation}
with uncertain temperature $T^0_x(\z) > T^0_y$. Trubnikov's formula (see Appendix \ref{sec:appendix} for further details) states
\begin{equation}
\Delta T(t, \z) \simeq \Delta T(0, \z) e^{-t/\tau(\z)},
\end{equation}
when $\Delta T(0, \z)$ is sufficiently small. The relaxation rate $\tau(\z)$ depends on the potential considered:
\begin{equation}
\tau(\z) = 
\begin{cases}
\dfrac{1}{8 C \rho} & \textrm{Maxwell case}\\
& \\
\dfrac{4 T^{3/2}(\z)}{C\rho\sqrt{\pi}} & \textrm{Coulomb case}.
\end{cases}
\end{equation}

In the Maxwell case ($\gamma=0$) we choose 
\begin{align} \label{eq:init_T_M}
&T^0_x(\z) = 0.7 + 0.1 \z,
\quad\textrm{with}\quad
\z\sim\mathcal{U}([0,1]), \\
&T^0_y = 0.5,
\end{align}
with $M=3$, and $C=1/2$. We investigate the trend to equilibrium with $N=30^2, 120^2$ particles. The left panel of Figure \ref{fig:test_5_trub} shows the evolution of the anisotropy of the temperature in time; as the number of particles increases, the \ac{sg} particle method recovers Trubnikov's behaviour.

In the Coulomb case ($\gamma=-3$) we fix the uncertain total temperature
\begin{equation}
T(\z) = 0.6 + 0.1\frac{\z}{2},
\quad\textrm{with}\quad
\z\sim\mathcal{U}([0,1]),
\end{equation}
which is conserved in time. We vary the temperatures along the y-axis in order to study different values of $\Delta T(0,\z)$. In particular, we fix $N=120^2$, $M=3$, $C=2$, and we consider two scenarios:
\begin{align} \label{eq:init_T_C}
(a) \quad & T^0_y = 0.3,
\quad\textrm{so that}\quad
T^0_x(\z) = 0.9 + 0.1\z,
\quad\textrm{with}\quad
\z\sim\mathcal{U}([0,1]); \\
(b) \quad & T^0_y = 0.5,
\quad\textrm{so that}\quad
T^0_x(\z) = 0.7 + 0.1\z,
\quad\textrm{with}\quad
\z\sim\mathcal{U}([0,1]). 
\end{align}
The right panel of Figure \ref{fig:test_5_trub} shows the evolution of the anisotropy of the temperature in time. The anisotropy tends to zero in time and, as expected, the agreement with Trubnikov's formula is better when the initial temperature difference is smaller.

\begin{figure}
\centering
\includegraphics[width = 0.45\linewidth]{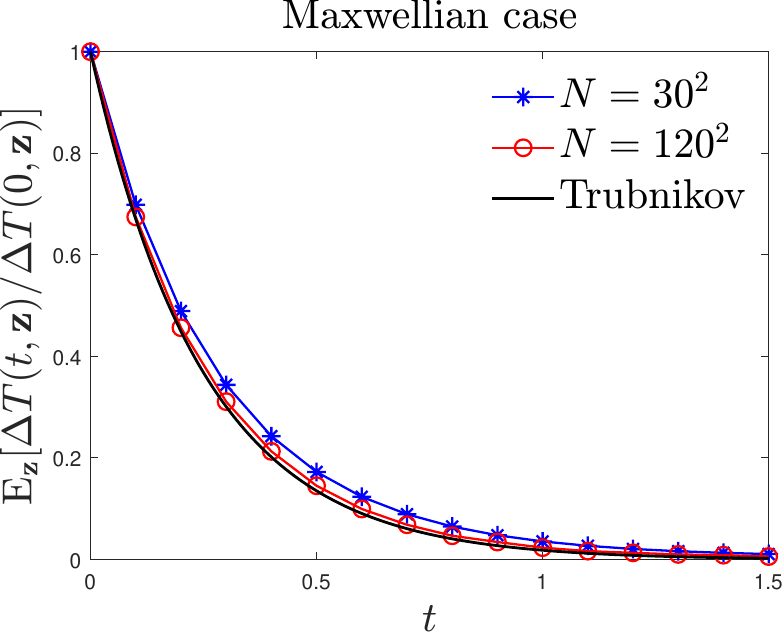}
\includegraphics[width = 0.45\linewidth]{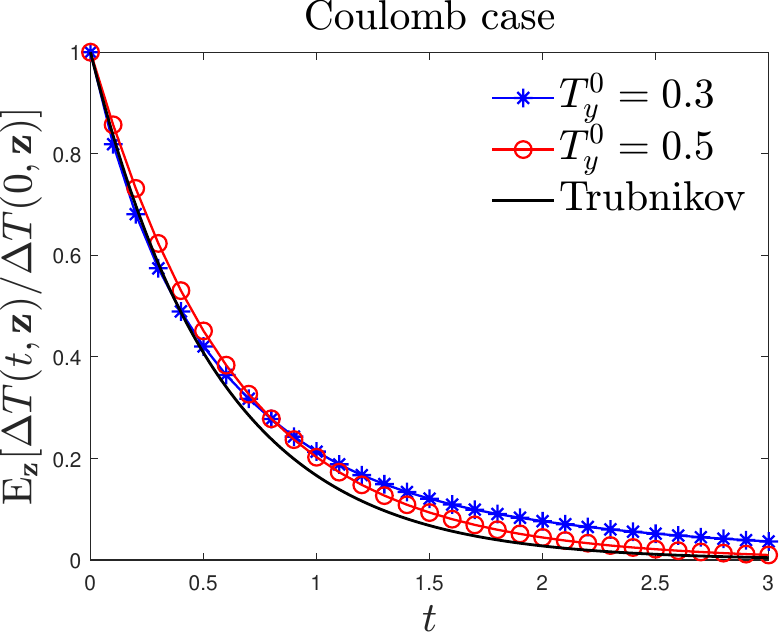}
\caption{\small{\textbf{Test 5}. Left: expectation of $\Delta T(t,\z)/\Delta T(0,\z)$ for Maxwellian molecules, i.e., $\gamma=0$. Initial conditions given by \eqref{eq:init_trubnikov} with initial temperatures \eqref{eq:init_T_M}. We choose $M=3$, $C=1/2$ and we investigate $N=30^2,120^2$. Trubnikov relaxation rate $\tau_{\textrm{M}}=\frac{1}{8 C \rho}$. Right: expectation of $\Delta T(t,\z)/\Delta T(0,\z)$ for Coulomb potential, i.e., $\gamma=-3$. Initial conditions given by \eqref{eq:init_trubnikov} with initial temperatures \eqref{eq:init_T_C}. We choose $M=3$, $C=2$, and $N=120^2$. Trubnikov relaxation rate $\tau_{\textrm{C}}(\z)=\frac{4 T^{3/2}(\z)}{C\rho\sqrt{\pi}}$.}}
\label{fig:test_5_trub}
\end{figure}

\subsection{Test 6: Gaussians on a triangle}

To conclude, we study the trend to equilibrium of the Coulomb model ($\gamma=-3$) from an initial condition consisting of the sum of three Gaussian with uncertain temperature ($d_\z=1$). 

We consider the initial datum
\begin{equation} \label{eq:init_triangle}
f^0(v,\z) = \frac{1}{3}\left( \mathcal{M}_{\rho,U_1,\bar{T}(\z)} + \mathcal{M}_{\rho,U_2,\bar{T}(\z)} + \mathcal{M}_{\rho,U_3,\bar{T}(\z)} \right) ,
\end{equation}
where $\bar{T}(\z)$ is the uncertain common variance of each Gaussian, and $U_1$, $U_2$, $U_3$ are the vertices of an equilateral triangle in velocity space which is inscribed in a circle of radius $d_r$ centred at the origin:
\begin{equation} \label{eq:init_tr_c}
U_1 = \left(0,\,d_r\right),
\quad
U_2 = \left(-\frac{d_r\sqrt{3}}{2},\, -\frac{d_r}{2}\right),
\quad
U_3 = \left(\frac{d_r\sqrt{3}}{2},\, -\frac{d_r}{2} \right).
\end{equation}

The total temperature of the system is $T(\z) = \bar{T}(\z) + d_r^2/2$, since
\begin{align}
T(\z) & = \frac{1}{2} \int_{\R^2} |v|^2 f^0(v,\z) \,\mathrm{d}v = \frac{1}{6} \int_{\R^2} |v|^2 \left( \mathcal{M}_{\rho,U_1,\bar{T}(\z)} + \mathcal{M}_{\rho,U_2,\bar{T}(\z)} + \mathcal{M}_{\rho,U_3,\bar{T}(\z)} \right) \,\mathrm{d}v \\
& = \frac{1}{2} \int_{\R^2} |v|^2 \mathcal{M}_{\rho,U_1,\bar{T}(\z)} \,\mathrm{d}v
= \bar{T}(\z) + d_r^2 - \frac{d_r^2}{2} = \bar{T}(\z) + \frac{d_r^2}{2}.
\end{align}
The asymptotic equilibrium is therefore the centred Maxwellian with mass $\rho$ and temperature $T(\z)$.

\begin{figure}
\centering
\includegraphics[width = 0.3\linewidth]{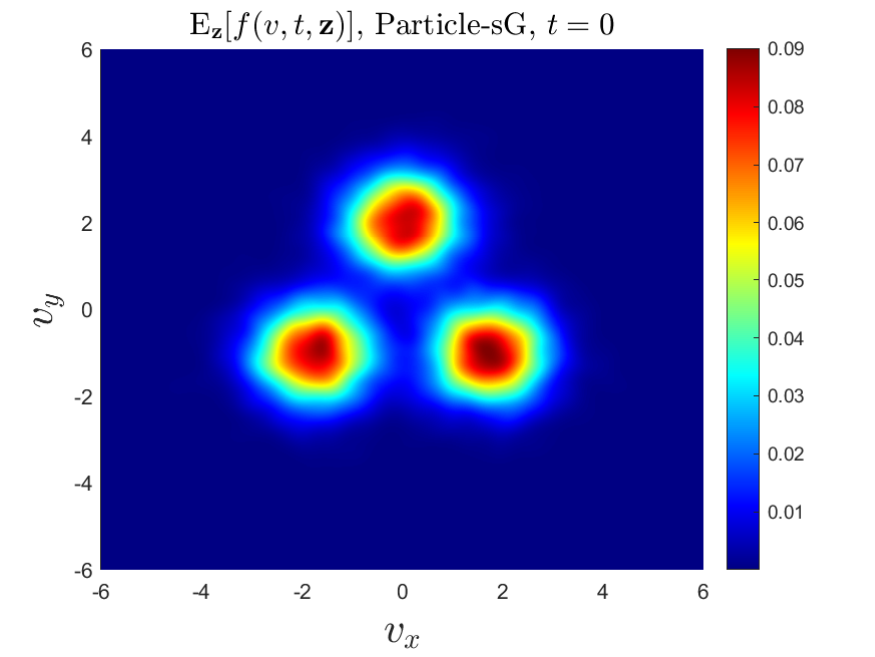}
\includegraphics[width = 0.3\linewidth]{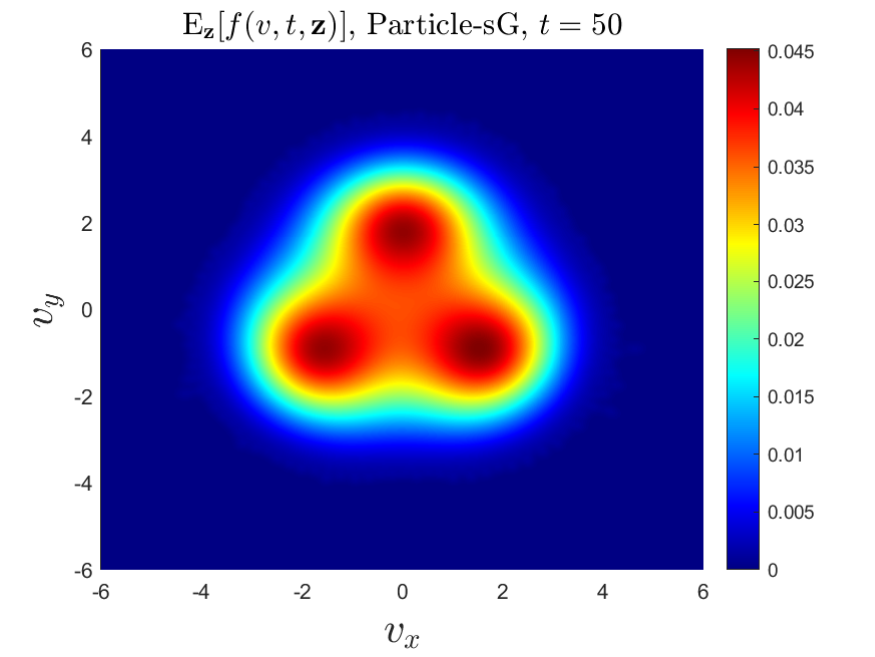}
\includegraphics[width = 0.3\linewidth]{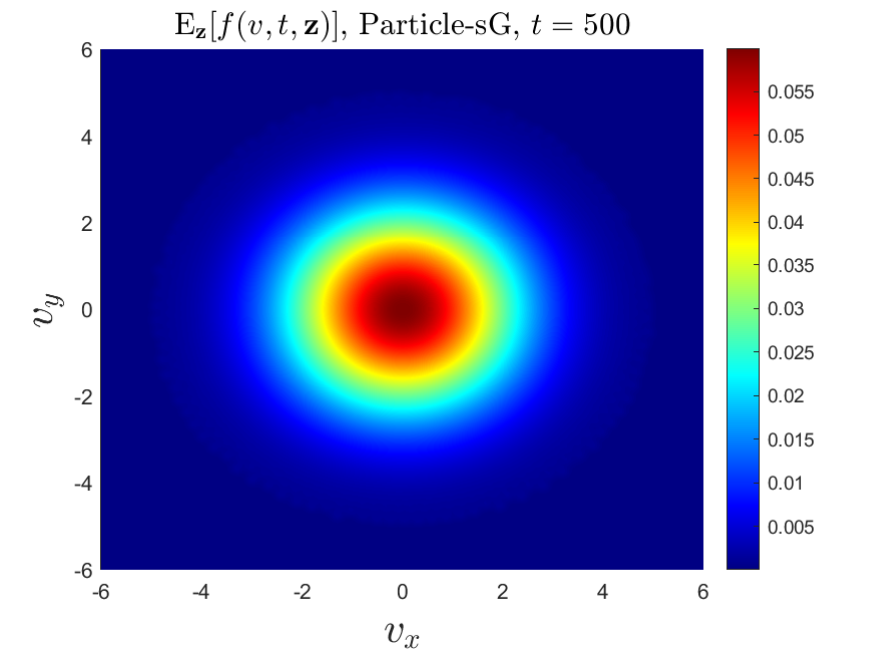}
\includegraphics[width = 0.3\linewidth]{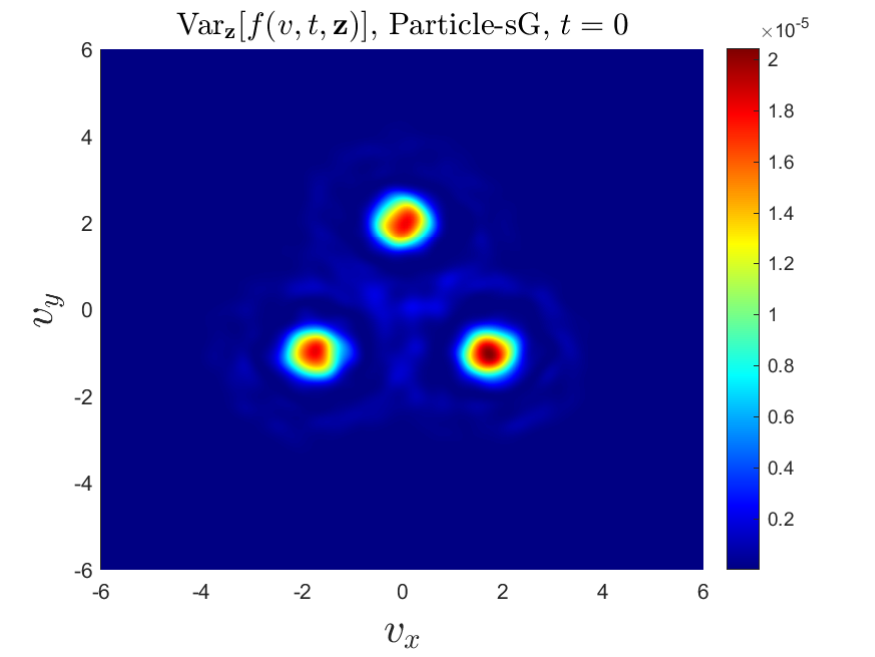}
\includegraphics[width = 0.3\linewidth]{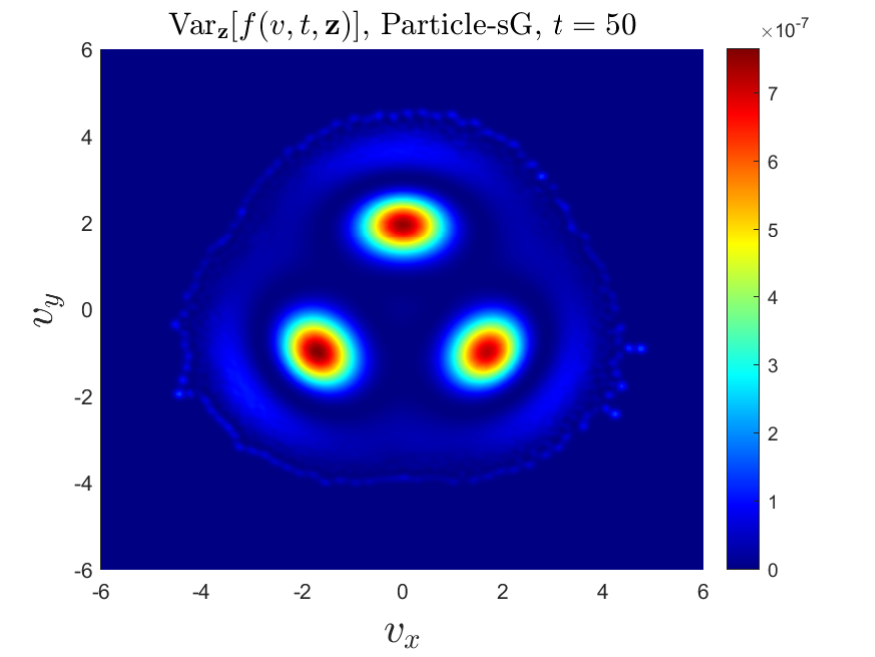}
\includegraphics[width = 0.3\linewidth]{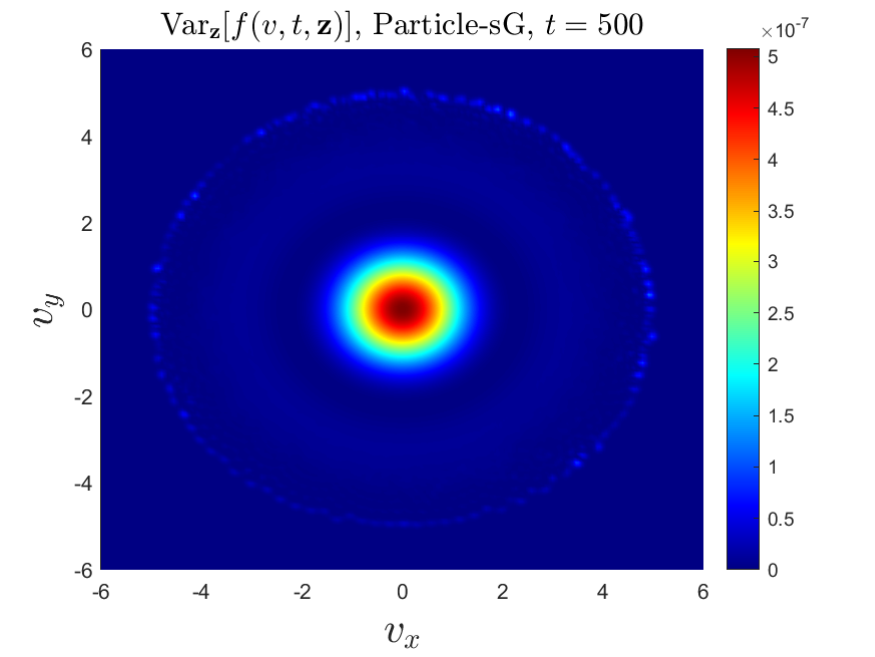}
\caption{\small{\textbf{Test 6}. Upper row: expected distributions $\mathbb{E}_{\z}[f(v,t,\z)]$ at times $t=0,50,500$ for the Gaussians on a triangle test. Lower row: variance of the distributions $\textrm{Var}_{\z}[f(v,t,\z)]$ at times $t=0,50,500$ for the same test. Numerical solution obtained with $N=120^2$ and $M=3$. Initial data given by \eqref{eq:init_triangle} with \eqref{eq:init_tr_c}, and $d_r=2$.}}
\label{fig:test_4}
\end{figure}

We choose $N=120^2$ particles and $M=3$. We fix $d_r=2$ and 
\begin{equation}
T(\z) = 0.5 + 0.1\z,
\quad\textrm{with}\quad
\z\sim\mathcal{U}([0,1]).
\end{equation}
Figure \ref{fig:test_4} shows the expectation $\mathbb{E}_{\z}[f(v,t,\z)]$ and the variance $\textrm{Var}_{\z}[f(v,t,\z)]$ of the solution at times $t=0,50,500$. The system reaches the correct equilibrium, the centred Gaussian with temperature $T(\z)$, and both the expectation and the variance are asymptotically radially symmetric.

\section*{Conclusions}

We have designed a \ac{sg} formulation of a deterministic particle method for the spatially homogeneous Landau equation with uncertainty. We have proven the structural properties of the \ac{sg} particle method (conservation of the invariant quantities and dissipation of the entropy), which hold pointwise in the uncertainty variable, and thus also in expectation. We have also shown the regularity of the method in the random space. We have validated the method in a range of numerical tests, involving uncertainty in the initial data, the interaction exponent $\gamma$, or both. We have demonstrated the spectral convergence of the method in suitable scenarios. We have shown the method's ability to capture known behaviours, such as the \ac{bkw} solution, and Trubnikov's formula. Future research directions will deal with the detailed analysis of the asymptotic propagation of regularity for the homogeneous model together with the design of a particle solver for the inhomogeneous Landau-Fokker-Planck equation in the presence of uncertain quantities.

\section*{Acknowledgments}
The work has been written within the activities of GNFM group of INdAM (National Institute of High Mathematics). 
RB and JAC were supported by the Advanced Grant Nonlocal-CPD (Nonlocal PDEs for Complex Particle Dynamics: Phase Transitions, Patterns and Synchronization) of the European Research Council Executive Agency (ERC) under the European Union’s Horizon 2020 research and innovation programme (grant agreement No.~883363) and by the EPSRC grant EP/T022132/1 ``Spectral element methods for fractional differential equations, with applications in applied analysis and medical imaging".
A.M. acknowledges partial support of the MIUR-PRIN2020 project No.2020JLWP23 and INdAM-GNFM project No. CUP E53C22001930001. 
M.Z. acknowledges partial support of the MIUR-PRIN2020 project No.2020JLWP23.

\appendix

\section{Trubnikov's formula for the Landau equation} \label{sec:appendix}

We derive Trubnikov's formula for the relaxation of anisotropic initial temperatures in dimension two ($d=2$), both for Maxwellian and Coulombian molecules. We follow the original contribution by Trubnikov \cite{Trubnikov1965} (Section 20, pages 200--203) and the derivation for Maxwellian molecules presented in \cite{medaglia2023} (Appendix A).

We consider the space homogeneous version of equation \eqref{eq:LFP} written in flux form,
\begin{equation}
\frac{\partial f(v,t)}{\partial t} = \nabla_v \cdot J(f)(v,t),
\end{equation}
with initial conditions given by
\begin{equation}
f_0(v)=\frac{\rho}{2\pi}\frac{1}{\sqrt{T_x T_y}} \exp\left\{-\dfrac{v^2_x}{2T_x}\right\} \exp\left\{-\dfrac{v^2_y}{2T_y}\right\},
\end{equation}
where $T_x >T_y $. Obviously $T=(T_x+T_y)/2$ is constant in time, so that
\begin{equation}
\frac{\mathrm{d}}{\mathrm{d}t} T_x = - \frac{\mathrm{d}}{\mathrm{d}t} T_y = - \frac{1}{\rho} \int_{\R^2} v^2_y \frac{\partial f}{\partial t} \,\mathrm{d}v = \frac{2}{\rho} \int_{\R^2} v_y J_y \,\mathrm{d}v,
\end{equation}
where the $y$-component of the flux $J_y$ reads
\begin{equation}
J_y = - C \int_{\R^2} \sum_{j\in\{x,y\}} |q|^{\gamma}(|q|^2 \delta_{yj} - q_y q_j)\left(f(v)\partial_{v_{*j}} f(v_*) - f(v_*) \partial_{v_{j}} f(v) \right) \,\mathrm{d}v_*,
\end{equation} 
for arbitrary $C>0$. Now, we exploit the knowledge of the initial distribution $f_0(v)$ to write explicitly the term inside the integral
\begin{equation} \label{eq:dtperp}
\frac{\mathrm{d}}{\mathrm{d}t} T_x = - \frac{\mathrm{d}}{\mathrm{d}t} T_y = - \frac{2C}{\rho} \frac{T_x - T_y}{T_x T_y} \underbrace{\int_{\R^2} \int_{\R^2} |q|^{\gamma}v_y q_y q^2_x f(v) f(v_*) \,\mathrm{d}v_* \,\mathrm{d}v}_{I}.
\end{equation}
To compute the integral $I$ in \eqref{eq:dtperp}, we perform the change of variables $(v,v_*)\to(V,q)$, where $V=(v+v_*)/2$ and $q=v-v_*$. In the general case, we also suppose that the temperature difference is small, i.e., $|T_x-T_y|\ll1$, so that $T_x\approx T$ and $T_y\approx T$
\begin{equation}
I = \int_{\R^2} \int_{\R^2} |q|^{\gamma}v_y q_y q^2_x f(V) f(q) \,\mathrm{d}V \,\mathrm{d}q 
\end{equation}
with
\begin{equation}
f(V) = \frac{\rho}{\pi T} \exp\left\{\dfrac{V^2}{T}\right\}, \qquad \textrm{and} \qquad f(q) = \frac{\rho}{4\pi T} \exp\left\{\dfrac{q^2}{4T}\right\}.
\end{equation}
As already observed in \cite{medaglia2023}, this condition can be relaxed in the Maxwell case, since we can compute $I$ independently from the magnitude of $|T_x-T_y|$.

Now we consider separately the Maxwell case ($\gamma=0$) and the Coulomb case ($\gamma=-3$). In the first scenario, we have
\begin{equation}
I_{\textrm{M}} = \int_{\R^2} \int_{\R^2} v_y q_y q^2_x f(V) f(q) \,\mathrm{d}V \,\mathrm{d}q = 2 \rho^2 T^2,
\end{equation}
in the second one
\begin{equation}
I_{\textrm{C}} = \int_{\R^2} \int_{\R^2} \frac{v_y q_y q^2_x}{|q|^3} f(V) f(q) \,\mathrm{d}V \,\mathrm{d}q = \frac{\rho^2\sqrt{\pi T}}{16}.
\end{equation}
Returning to \eqref{eq:dtperp}, and observing that
\begin{equation}
\frac{\mathrm{d}}{\mathrm{d}t} \Delta t =
\frac{\mathrm{d}}{\mathrm{d}t} (T_x - T_y) =
2\frac{\mathrm{d}}{\mathrm{d}t} T_x,
\end{equation}
we obtain $\Delta T(t) = \Delta T(0) e^{-t/\tau}$, where the relaxation parameter is $\tau_{\textrm{M}} = \frac{1}{8 C \rho}$ for Maxwellian molecules, and $\tau_{\textrm{C}} = \frac{4 T^{3/2}}{C\rho\sqrt{\pi}}$ for the Coulomb case. Consistently with the results for the three dimensional case, the relaxation parameter is only independent of the temperature for Maxwellian particles.

\bibliographystyle{abbrv}
\bibliography{Landau_SGDP.bib}

\begin{thebibliography}{10}

\bibitem{ambrosio2005gradient}
L.~Ambrosio, N.~Gigli, and G.~Savar{\'e}.
\newblock {\em Gradient flows: in metric spaces and in the space of probability
  measures}.
\newblock Springer Science \& Business Media, 2005.

\bibitem{birdsall85}
C.~K. Birdsall and A.~B. Langdon.
\newblock {\em Plasma Physics via Computer Simulation}.
\newblock McGraw-Hill, 1985.

\bibitem{bobylev2000}
A.~V. Bobylev and K.~Nanbu.
\newblock Theory of collision algorithms for gases and plasmas based on the
  {B}oltzmann equation and the {L}andau-{F}okker-{P}lanck equation.
\newblock {\em Phys. Rev. E}, 61(4):4576, 2000.

\bibitem{Dimarco2008}
R.~Caflisch, C.~Wang, G.~Dimarco, B.~Cohen, and A.~Dimits.
\newblock A hybrid method for accelerated simulation of {C}oulomb collisions in
  a plasma.
\newblock {\em Multiscale Model. Simul.}, 7(2):865--887, 2008.

\bibitem{carrillo2019blob}
J.~A. Carrillo, K.~Craig, and F.~S. Patacchini.
\newblock A blob method for diffusion.
\newblock {\em Calc. Var. Partial Differ. Equ.}, 58:1--53, 2019.

\bibitem{carrillo2022boltzmann}
J.~A. Carrillo, M.~G. Delgadino, and J.~Wu.
\newblock Boltzmann to {L}andau from the gradient flow perspective.
\newblock {\em Nonlinear Anal.}, 219:112824, 2022.

\bibitem{carrillo2023convergence}
J.~A. Carrillo, M.~G. Delgadino, and J.~S. Wu.
\newblock Convergence of a particle method for a regularized spatially
  homogeneous {L}andau equation.
\newblock {\em Math. Models Methods Appl. Sci.}, 33(05):971--1008, 2023.

\bibitem{Jingwei2020}
J.~A. Carrillo, J.~Hu, L.~Wang, and J.~Wu.
\newblock A particle method for the homogeneous {L}andau equation.
\newblock {\em J. Comput. Phys. X}, 7:100066, 24, 2020.

\bibitem{Carrillo2022}
J.~A. Carrillo, S.~Jin, and Y.~Tang.
\newblock Random batch particle methods for the homogeneous {L}andau equation.
\newblock {\em Commun. Comput. Phys.}, 31(4):997--1019, 2022.

\bibitem{carrillo2010nonlinear}
J.~A. Carrillo, S.~Lisini, G.~Savar{\'e}, and D.~Slep{\v{c}}ev.
\newblock Nonlinear mobility continuity equations and generalized displacement
  convexity.
\newblock {\em J. Funct. Anal.}, 258(4):1273--1309, 2010.

\bibitem{carrillo2003kinetic}
J.~A. Carrillo, R.~J. McCann, and C.~Villani.
\newblock Kinetic equilibration rates for granular media and related equations:
  entropy dissipation and mass transportation estimates.
\newblock {\em Rev. Mat. Iberoam.}, 19(3):971--1018, 2003.

\bibitem{CarrilloPZ19}
J.~A. Carrillo, L.~Pareschi, and M.~Zanella.
\newblock Particle based g{PC} methods for mean-field models of swarming with
  uncertainty.
\newblock {\em Commun. Comput. Phys.}, 25(2):508--531, 2019.

\bibitem{Carrillo19}
J.~A. Carrillo and M.~Zanella.
\newblock Monte {C}arlo g{PC} methods for diffusive kinetic flocking models
  with uncertainties.
\newblock {\em Vietnam J. Math.}, 47(4):931--954, 2019.

\bibitem{Crouseilles2004}
N.~Crouseilles and F.~Filbet.
\newblock Numerical approximation of collisional plasmas by high order methods.
\newblock {\em J. Comput. Phys.}, 201(2):546--572, 2004.

\bibitem{crouseilles2010}
N.~Crouseilles, M.~Mehrenberger, and E.~Sonnendr{\"u}cker.
\newblock Conservative semi-{L}agrangian schemes for {V}lasov equations.
\newblock {\em J. Comput. Phys.}, 229(6):1927--1953, 2010.

\bibitem{Dimarco2010}
G.~Dimarco, R.~Caflisch, and L.~Pareschi.
\newblock Direct simulation {M}onte {C}arlo schemes for {C}oulomb interactions
  in plasmas.
\newblock {\em Commun. Appl. Ind. Math.}, 1(1):72--91, 2010.

\bibitem{Dimarco2015}
G.~Dimarco, Q.~Li, L.~Pareschi, and B.~Yan.
\newblock Numerical methods for plasma physics in collisional regimes.
\newblock {\em J. Plasma Phys.}, 81(1):1--31, 2015.

\bibitem{Dimarco22_2}
G.~Dimarco, L.~Pareschi, and M.~Zanella.
\newblock Micro-macro stochastic {G}alerkin methods for nonlinear
  {F}okker-{P}lank equations with random inputs.
\newblock {\em Multiscale Model. Simul.}, In press.

\bibitem{erbar2023gradient}
M.~Erbar.
\newblock A gradient flow approach to the {B}oltzmann equation.
\newblock {\em J. Eur. Math. Soc.}, 2023.

\bibitem{Filbet2001}
F.~Filbet, E.~Sonnendr{\"u}cker, and P.~Bertrand.
\newblock Conservative {N}umerical {S}chemes for the {V}lasov {E}quation.
\newblock {\em J. Comput. Phys.}, 172(1):166--187, 2001.

\bibitem{gamba2017fast}
I.~M. Gamba, J.~R. Haack, C.~D. Hauck, and J.~Hu.
\newblock A fast spectral method for the {B}oltzmann collision operator with
  general collision kernels.
\newblock {\em SIAM J. Sci. Comput.}, 39(4):B658--B674, 2017.

\bibitem{gerritsma2010time}
M.~Gerritsma, J.-B. Van~der Steen, P.~Vos, and G.~Karniadakis.
\newblock Time-dependent generalized polynomial chaos.
\newblock {\em J. Comput. Phys.}, 229(22):8333--8363, 2010.

\bibitem{GZ2017}
M.~Gualdani and N.~Zamponi.
\newblock {Spectral gap and exponential convergence to equilibrium for a
  multi-species Landau system}.
\newblock {\em Bulletin des Sciences Math{\'e}matiques}, 141:509--538, 2017.

\bibitem{Jin2017}
S.~Jin and L.~Pareschi, editors.
\newblock {\em Uncertainty Quantification for Hyperbolic and Kinetic
  Equations}, volume~14 of {\em SEMA-SIMAI Springer Series}.
\newblock Springer, 2017.

\bibitem{liu2018hypocoercivity}
L.~Liu and S.~Jin.
\newblock Hypocoercivity based sensitivity analysis and spectral convergence of
  the stochastic galerkin approximation to collisional kinetic equations with
  multiple scales and random inputs.
\newblock {\em Multiscale Model. Simul.}, 16(3):1085--1114, 2018.

\bibitem{medaglia2023}
A.~Medaglia, L.~Pareschi, and M.~Zanella.
\newblock Particle simulation methods for the landau-fokker-planck equation
  with uncertain data.
\newblock {\em Preprint arXiv:2306.07701}, 2023.

\bibitem{medaglia2022JCP}
A.~Medaglia, L.~Pareschi, and M.~Zanella.
\newblock Stochastic {G}alerkin particle methods for kinetic equations of
  plasmas with uncertainties.
\newblock {\em J. Comput. Phys.}, 479:112011, 2023.

\bibitem{Medaglia2022}
A.~Medaglia, A.~Tosin, and M.~Zanella.
\newblock Monte {C}arlo stochastic {G}alerkin methods for non-{M}axwellian
  kinetic models of multiagent systems with uncertainties.
\newblock {\em Partial Differ. Equ. Appl.}, 3:51, 2022.

\bibitem{Pareschi2021}
L.~Pareschi.
\newblock An introduction to uncertainty quantification for kinetic equations
  and related problems.
\newblock In {\em Trails in {K}inetic {T}heory: {F}oundational {A}spects and
  {N}umerical {M}ethods}, volume~25 of {\em SEMA SIMAI Springer Ser.}, pages
  141--181. Springer, Cham, 2021.

\bibitem{Pareschi2000}
L.~Pareschi, G.~Russo, and G.~Toscani.
\newblock Fast spectral methods for the {F}okker-{P}lanck-{L}andau collision
  operator.
\newblock {\em J. Comput. Phys.}, 165(1):216--236, 2000.

\bibitem{Pareschi2020}
L.~Pareschi and M.~Zanella.
\newblock Monte {C}arlo stochastic {G}alerkin methods for the {B}oltzmann
  equation with uncertainties: {S}pace-homogeneous case.
\newblock {\em J. Comput. Phys.}, 423:1098--1120, 2020.

\bibitem{pinto2014charge}
M.~C. Pinto, S.~Jund, S.~Salmon, and E.~Sonnendr{\"u}cker.
\newblock Charge-conserving {FEM}--{PIC} schemes on general grids.
\newblock {\em C. R. Mec.}, 342(10-11):570--582, 2014.

\bibitem{RRCD}
L.~F. Ricketson, M.~S. Rosin, R.~Caflisch, and A.~M. Dimits.
\newblock An entropy based thermalization scheme for hybrid simulations of
  {C}oulomb collisions.
\newblock {\em J. Comput. Phys.}, 273:77--99, 2014.

\bibitem{sonnendrucker1999semi}
E.~Sonnendr{\"u}cker, J.~Roche, P.~Bertrand, and A.~Ghizzo.
\newblock The semi-{L}agrangian method for the numerical resolution of the
  {V}lasov equation.
\newblock {\em J. Comput. Phys.}, 149(2):201--220, 1999.

\bibitem{Trubnikov1965}
B.~Trubnikov.
\newblock Particle interaction in a fully ionized plasma.
\newblock {\em Rev. Plasma Phys.}, 1:105--204, 1965.

\bibitem{xiao2021stochastic}
T.~Xiao and M.~Frank.
\newblock A stochastic kinetic scheme for multi-scale plasma transport with
  uncertainty quantification.
\newblock {\em J. Comput. Phys.}, 432:110139, 2021.

\bibitem{Xiu2010}
D.~Xiu.
\newblock {\em Numerical {M}ethods for {S}tochastic {C}omputations}.
\newblock Princeton University Press, 2010.

\bibitem{Xiu2002}
D.~Xiu and G.~E. Karniadakis.
\newblock The {W}iener--{A}skey polynomial chaos for stochastic differential
  equations.
\newblock {\em SIAM J. Sci. Comput.}, 24(2):619--644, 2002.

\end{thebibliography}

\end{document}